\documentclass[final,leqno]{siamltex}
\usepackage{mathrsfs}
\usepackage{amsfonts}
\usepackage{enumerate}
\usepackage{amsmath}
\usepackage{amssymb}
\usepackage{graphicx}
\usepackage{subfigure}
\usepackage{subfigure,xcolor}

\date{}

\newtheorem{remark}{Remark} [section]
\newtheorem{example}{Example}[section]
\newcommand{\IR}{{\mathbb{R}}}

\newcommand{\IN}{{\mathbb{N}}}

\def\dsum{\displaystyle\sum}

\title{Equilibrium Problems on Riemannian Manifolds with Applications}

\author{Chong Li\thanks{School of Mathematical Sciences, Zhejiang University, Hangzhou 310027,
P. R. China (cli@zju.edu.cn). Research of this author was supported in part by the National Natural Science
Foundation of China (grant 11571308).}
\and Xiangmei Wang\thanks{College of Mathematics and Statistics, Guizhou University, Guiyang 550025, P. R. China (xmwang2@gzu.edu.cn). Research of this author was supported in part by the National Natural Science
Foundation of China (grant 11661019) and the Natural Science Foundation of Guizhou Province, China (grant 20161039).} \and Genaro L\'{o}pez\thanks{Department of Mathematical Analysis,
University of Sevilla. 1160, 41080-Sevilla, Spain
(glopez@us.es). Research of this author was supported in part by the DGES, grant MTM2015-65242-C2-1P
and Junta de Andaluc\'{\i}a,
grant FQM-127.}
\and Jen-Chih Yao\thanks{China Medical University, Taichung, Taiwan 40402, R.O.C. and Department of Mathematics, King Abdulaziz University, P.O. Box 80203, Jeddah 21589, Saudi Arabia (yaojc@mail.cmu.edu.tw). Research of this author was supported in part by the Grant MOST 102-2115-M-039-003-MY3.}}

\begin{document}
\maketitle
\slugger{mms}{xxxx}{xx}{x}{x--x}

\begin{abstract} We study the equilibrium problem on general Riemannian manifolds. The results on existence of solutions and on the convex structure of the solution set  are established. Our approach consists  in relating the equilibrium problem to a suitable variational inequality problem on Riemannian manifolds, and is completely different from previous ones on this topic in the literature.
As applications,   the corresponding    results for  the mixed variational inequality and the Nash equilibrium are obtained. Moreover, we formulate and analyze the convergence of  the proximal point algorithm for the equilibrium problem. In particular,   correct proofs are provided for the results claimed in J. Math. Anal. Appl. 388, 61-77, 2012  (i.e., Theorems 3.5 and 4.9 there)   regarding the existence of the mixed variational inequality and the domain of the resolvent for  the equilibrium problem  on Hadamard manifolds.
\end{abstract}

\begin{keywords} Riemannian manifold, equilibrium problem, variational inequality problem, proximal point algorithm\end{keywords}

\begin{AMS} Primary, 90C25; Secondary, 65K05\end{AMS}

\section{Introduction}

Let $X$ be a metric space, $Q\subseteq X$ a nonempty subset and $F:X\times X\rightarrow(-\infty,+\infty]$ a bifunction. The equilibrium problem (introduced by Blum and Oettli in \cite{Blum1994}), abbreviated as EP, associated to the pair ($F$, $Q$) is to find a point $\bar x\in Q$ such that
\begin{equation}\label{EP}
 F(\bar x,y)\ge0\quad\mbox{for any } y\in Q.
\end{equation}
As shown in \cite{Blum1994,Oettli1997}, EP contains, as special cases, optimization problems, complementarity problems, fixed point problems, variational inequalities and problems of Nash equilibria; and it has been broadly applied in many areas, such as economics, image reconstruction, transportation,  network, and elasticity. In recent years, EP has been studied extensively,  including the issues regarding existence  of solutions   and iterative algorithms for finding solutions;  see e.g., \cite{Bianchi1996,Blum1994,Combettes2005,Flam1997,Iusem2003b}.

Since the classical existence results in EPs work for the case when $Q$ is a convex set and the bifunction $F$ is convex
in the second variable, 
some authors focused on exploiting the possible existence without the convexity assumption; see e.g.,
 \cite{Homidan2008,Bianchi2005,Castellani2010,Iusem2009}. One useful approach   used  in \cite{Kristaly2010} and \cite{Kristaly2014}
  is to  embed  the underlying nonconvex and/or nonsmooth Nash/Nash-type equilibrium problems into a suitable Riemannain manifold $M$
  to study the existence and the location problems  of the Nash/Nash-type equilibrium points. In particular, Krist\'{a}ly established in  \cite{Kristaly2010},
the existence results for  Nash  equilibrium points associated to strategy sets $\{Q_i\subseteq M_i\}_{i\in I}$ and loss-functions $\{f_i\}_{i\in I}$ ($I:=\{1,2,\ldots,n$\}) under the following assumption:
\begin{enumerate}
  \item[$\rm (A_K$):]  each $Q_i$ is a compact and geodesic convex set of $M_i$ for all $i\in I$
\end{enumerate}
(see item (e) in Definition \ref{convexset} for   the notion of the geodesic convexity).
This class of approaches has also been used extensively in many optimization problems since some nonconvex and/or nonsmooth
problems of the constrained optimization in $\IR^n$ can be reduced to
convex and/or smooth unconstrained optimization problems on appropriate
Riemannian manifolds; see, for examples, {\cite{Ferreira2005,Mahony1996,Miller2005,Smith1994,Udriste1994}.} More about optimization techniques and notions in Riemannian manifolds can be found in {\cite{Absil2008,Adler2002,Burke2001,Greene-Wu,Li2009,LiMWY2011,LiY2012,Wang2010,Yang2007}}
and the bibliographies therein.

 {For the equilibrium problem \eqref{EP} on  a Riemannian manifold $M$,  Colao et. al, by generalizing the KKM lemma to a Hadamard manifold, established   an   existence  result (i.e., \cite[Theorem 3.2]{Colao2012}) for solutions of EP under the following assumptions:}

\begin{enumerate}
  \item[$\rm (A_C$-1):]  $M$ is a Hadamard manifold and $Q$ is closed and convex;
  \item[$\rm (A_C$-2):]  the set $\{y\in Q:F(x,y)\le0\}$ is convex for any $x\in Q$
\end{enumerate}
%
{(which was extensively studied in \cite{Batista2015} for the generalized vector equilibrium problem).
This} existence result was applied there to solve  the   following problems:
\begin{enumerate}
  \item[(P1).] the   existence problem of solutions for mixed variational inequality problems;
  \item[(P2).] the well-definedness of the resolvent and the proximal point algorithm for solving EP;
  \item[(P3).] the existence problem of  fixed points   for set-valued mappings;
  \item[(P4).] the existence problem of solutions for Nash equilibrium problems.
\end{enumerate}
However, the applications to problems  (P1)-(P3) above rely heavily
on the following claim: 
\begin{equation}\label{f}
\mbox{the function $y\mapsto \langle u_{x},\exp_{x}^{-1}y\rangle$  is quasi-convex},
\end{equation}
where $x\in M$ and $u_x\in T_{x}M$;   see the proofs for Theorems 3.5, 3.10 and 4.9 in  \cite{Colao2012}. Unfortunately, unlike in the linear space setting,
claim \eqref{f} is not true in general as pointed out in \cite[Theorem 2.1, p. 299]{Udriste1994} or \cite{KLLN2015}. { Note that, for any $x\in M$, the function defined by \eqref{f} is convex at $x$ (see Definition \ref{definition-convex-f} (i)); this motivates us to introduce the new notion of the point-wise (weak) convexity for a bifunction on general manifolds (see Definition \ref{monotone-bif} (c)).}

Our main purpose in the present  paper is to develop a new approach {(based on the new notion and the work in \cite{LiY2012})} to study  the issue on the existence and structure
of solutions for equilibrium problems on general Riemannian manifolds, which, in  particular,   covers problems (1), (2) and (4) as special cases.
%
In our approach, rather than assumptions  $\rm (A_K)$ or $\rm (A_C$-1)-$\rm (A_C$-2), we make the following ones  on the involved $Q$ and $F$:

$\bullet$ $Q$ is a closed and weakly convex subset of Riemannian manifold $M$;

$\bullet$  $F$ is point-wise weakly convex on $Q$.\\


The technique used  in the present paper for proving the main results  is completely different from the ones used in \cite{Colao2012,Batista2015,Kristaly2010}. Actually, 
our technique here is mainly focused on
establishing  the equivalence between the {\rm EP} and a suitable
variational inequality problem; and then apply the corresponding results in \cite{LiY2012} for the variational inequality problem to study the existence of solutions and the convex structure of the solution set of the {\rm EP} .  

As applications to problems (P1), (P2) and (P4), we obtained  some  results on  the existence of solutions and convexity of the solution sets
 for   mixed variational inequality problems and Nash equilibrium problems (see Theorem 5.1 and 5.2), as well as the convergence of the proximal point algorithm for solving EP. In particular, the existence result for     mixed variational inequality problems and the well-definedness results of the resolvent  for solving EP on Hadamard manifolds  provide correct proofs for the corresponding ones   \cite[Theorem 3.5 and 4.9]{Colao2012} (see the explanations before Corollaries \ref{APA-coro} and 5.2 in Section 4 and 5, respectively); while   the existence result  for the Nash equilibrium on general manifolds relaxes  the geodesic convexity assumption made on  $\{Q_i\}$
   in \cite[Theorem 1.1]{Kristaly2010}  to the weaker one that each  $ Q_i $ is weakly  convex. It is worthwhile to notice that the geodesic convexity assumption for  $\{Q_i\}$ in \cite{Kristaly2010} 
{prevents} its application to  some special but important Riemannian manifolds, such as compact Stiefel manifolds ${\rm St}(p,n)$ and Grassmann manifolds ${\rm Grass}(p,n)$ ($p<n$), in which there is no geodesic convex subset; see Remark 5.1 in Section 5. Moreover,  to our best knowledge,
the convex structure results on  the solution set  for mixed variational inequality problems and Nash equilibrium problems are new even in Hadamard manifold settings.

The paper is organized as follows. In the next section, we introduce some basic notions and notations on Riemannian manifolds, some properties about the (weakly) convex function and the results about the VIP in \cite{LiY2012} which will be used in our approach. In section 3, we establish the existence and the uniqueness result of the solution and the convexity of the solution set of the EP on general Riemannian manifolds. Following these, the formulation of the proximal point algorithm for the equilibrium problem on general Riemnannian manifolds is given and the convergence property about the algorithm is analyzed in section 4. The last section is devoted to  the applications to the Nash equilibrium problem and the mixed variational inequality problem.

\section{Notations and preliminary results}
\label{sec:2}


\subsection{Background of Riemannian manifolds}
The notations used in the present paper are standard;
and the readers are referred to some textbooks for more details, for example,
\cite{Carmo1992,Sakai1996,Udriste1994}.

Let $M$ be a connected $n$-dimensional Riemannian manifold with the Levi-Civita connection   $\nabla$
  on $M$. Let $x\in M$, and  let $T_{x}M$ stand for  the
tangent space at $x$ to $M$ endowed with  the scalar product $\langle,\rangle_{x}$
 and  the associated norm $\|.\|_{x}$, where the subscript $x$ is
sometimes omitted. Thus the tangent bundle,  denoted by $TM$, is defined by
{$$TM:=\{(x,v):x\in M,v\in T_xM\}.$$}
 Fix $y\in{M}$, and let $\gamma:[0,1]\rightarrow M$ be
a piecewise smooth curve joining $x$ to $y$. Then, the arc-length of
$\gamma$ is defined by $l(\gamma):=\int_{0}^{1}\|\dot{\gamma}(t)\|dt$,
while the Riemannian distance from $x$ to $y$ is defined by ${\rm
d}(x,y):=\inf_{\gamma}l(\gamma)$, where the infimum is taken over
all piecewise smooth curves $\gamma:[0,1]\rightarrow M$ joining $x$
to $y$. We use ${\bf B}(x,r)$ and $\overline{{\bf B}(x,r)}$ to
denote, respectively, the open metric ball and the closed metric ball at $x$ with radius
$r$, that is,
$${\bf B}(x,r):=\{y\in M:{\rm d}(x, y)<r\} \quad \mbox{and}\quad \overline{{\bf B}(x,r)}:=\{y\in M:{\rm d}(x, y)\le r\} .$$

A vector field $V$ is said to be parallel along $\gamma$  if
$\nabla_{\dot{\gamma}}V=0$. In particular, for a smooth curve $\gamma$,
if $\dot{\gamma}$ is parallel along itself, then  $\gamma$ is called a
geodesic, that is, a smooth curve $\gamma$ is a geodesic if an only
if $\nabla_{\dot{\gamma}}{\dot{\gamma}}=0$. A geodesic
$\gamma:[0,1]\rightarrow M$ joining $x$ to $y$ is minimal if its
arc-length equals its Riemannian distance between $x$ and $y$. By
the Hopf-Rinow theorem \cite{Carmo1992}, if $M$ is complete, then $(M,{\rm d})$ is a complete
metric space, and there is at least one minimal geodesic joining $x$
to $y$.  One of the important structures on $M$ is the exponential map
$\exp_x:T_xM\rightarrow M$, which is defined at $x\in M$ by
$\exp_xv=\gamma_v(1,x)$ for each $v\in T_xM$, where
$\gamma_v(\cdot,x)$ is the geodesic starting at $x$
with velocity $v$. Then, $\exp_xtv=\gamma_v(t,x)$ for each real
number $t$. Another useful tool is  the parallel transport $P_{\gamma,\cdot,\cdot}$ on the tangent
bundle $TM$ along a geodesic $\gamma$, which is defined by
$$P_{\gamma,\gamma(b),\gamma(a)}(v)=V(\gamma(b))\quad\mbox{for any $a,b\in\IR$ and $v\in T_{\gamma(a)}M$},$$
where $V$ is the unique vector field satisfying $V(\gamma(a))=v$ and $\nabla_{\dot{\gamma}(t)}V=0$ for all $t$.
Then, for any $a,b\in\IR$, $P_{\gamma,\gamma(b),\gamma(a)}$ is an isometry from $T_{\gamma(a)}M$ to $T_{\gamma(b)}M$. We will
write $P_{y,x}$ instead of $P_{\gamma,y,x}$ in the case when $\gamma$ is a minimal geodesic joining $x$ to $y$ and no confusion arises.
 The following lemma can be checked easily.

\begin{lemma}\label{cvg-exp}
Let $x_0\in M$ and $\{x_k\}\subset M$ be such that $\lim_{k\rightarrow\infty}x_k=x_0$. Let $u_0,v_0\in T_{x_0}M$  and let $\{u_k\}$,   $\{v_k\}$ be sequences with each $u_k,v_k\in T_{x_k}M$   such that $u_k\rightarrow u_0$ and $v_k\rightarrow v_0$.
Then
$$
\exp_{x_k}u_k\rightarrow\exp_{x_0}u_0\quad\mbox{and}\quad \langle u_k,v_k\rangle\rightarrow\langle u_0,v_0\rangle.
$$
\end{lemma}
%
%

The following result is known  in any textbook about Riemannain geometry, see e.g., \cite[Corollary, p. 73]{Carmo1992} or \cite[Exercise 5, p. 39]{Sakai1996}.

\begin{lemma}\label{minimal-geodesic} Let $\gamma:[a,b]\rightarrow M$ be
  a piecewise differentiable curve. If $l(\gamma)={\rm d}(\gamma(a),\gamma(b))$, then $\gamma$ is a geodesic joining  $\gamma(a)$ and $\gamma(b)$.
\end{lemma}

Consider a set $Q\subseteq M$ and $x,\,y\in Q$. The set of
all geodesics $\gamma:[0,1]\rightarrow M$ with $\gamma(0)=x$ and
$\gamma(1)=y$ satisfying $\gamma([0,1])\subseteq Q$ is denoted by   $\Gamma^Q_{xy}$, that is
$$\Gamma^Q_{xy}:=\{\gamma:[0,1]\rightarrow Q:\;\gamma(0)=x,\, \gamma(1)=y\mbox{ and } \nabla_{\dot{\gamma}}\dot{\gamma}=0\}.
$$
In particular, we write  $\Gamma_{xy}$ for $\Gamma^M_{xy}$, and  $\Gamma_{xy}$ is nonempty for all $x,y\in M$ provided that $M$ is complete. Furthermore,
 for a subset $\Gamma_0\subseteq \Gamma_{xy}$, we use ${\rm min-}\Gamma_0$ to denote the subset of $\Gamma_0$ consisting of all minimal geodesics in $\Gamma_0$. Thus       $\gamma_{xy}\in {\rm min-}\Gamma_0$  means that $\gamma_{xy}\in \Gamma_0$ and $\gamma_{xy}$ is minimal.

Recall that a   Hadamard manifold  is a complete simply connected $m$-dimensional
Riemannian manifold with nonpositive sectional curvatures. In
a Hadamard manifold, the geodesic between any two points is unique and the exponential
map at each point of $M$ is a global diffeomorphism; see, e.g., \cite[Theorem 4.1, p. 221]{Sakai1996}.
Thus ${\rm min-}\Gamma_{xy}$ coincides with $ \Gamma_{xy}$ in a  Hadamard manifold for any $ x,y\in M$.

\subsection{Convex analysis on Riemmanian manifolds}


Definition \ref{convexset} below presents the notions of the convexity for subsets in $M$, where item (e) is known in  \cite{Kristaly2010}, and see e.g., \cite{LiLi2009,Wang2010} for  the others. As usual, we use $\overline{C}$ to stand for the closure of a subset $C\subseteq M$.

\begin{definition}\label{convexset} Let $Q\subseteq M$ be
a  nonempty set. The set $Q$ is said to be

{\rm (a)} weakly convex if, for any $x,y\in Q$, there is a minimal
geodesic of $M$ joining $x$ to $y$ and it is in $Q$;

{\rm (b)} strongly convex if, for any $x,y\in Q$,  the
minimal geodesic  in  $M$ joining $x$ to $y$ is unique and lies in $Q$;

{\rm (c)} locally convex if, for any $x\in \overline{Q}$, there is a positive $\varepsilon>0$ such that $Q\cap{\bf B}(x,\varepsilon)$ is strongly convex;

{\rm (d)} $r$-convex if, for any $x,y\in Q$ with ${\rm d}(x,y)\le r$, the
minimal geodesic  in  $M$ joining $x$ to $y$ is unique and lies in $Q$;

{\rm (e)} geodesic convex if,  for any $x,y\in Q$, the
geodesic in  $M$ joining $x$ to $y$ is unique and lies in $Q$.
\end{definition}

\begin{remark} {\rm (a)} The following implications are obvious: 
{$$
\mbox{geodesic convexity $\Rightarrow$} \mbox{strong convexity $\Rightarrow$} \mbox{ $r$-convexity/weak convexity }
\mbox{$\Rightarrow$ local convexity.}
$$}

{\rm (b)} The  intersection of a weakly convex set and a strongly convex set is strongly convex.

{\rm (c)}  All   convexities (except the local convexity)
in a Hadamard manifold coincide and are simply called the convexity.
\end{remark}

Recall that the convexity radius at $x$ is defined by
\begin{equation}\label{cvx-rd-p}
r_x:=\sup\left\{r>0:\begin{array}{ll}
&\mbox{each ball in ${\bf B}(x,r)$ is strongly convex}\\
&\mbox{and each geodesic in ${\bf B}(x,r)$ is minimal} \end{array}\right\}.
\end{equation}
Then  $r_x$ is well defined and positive, and $r_x=+\infty$ for each $x\in M$ in the case when  $M$ is a Hadamard manifold. Moreover, for any compact subset $Q\subseteq M$, we have that
\begin{equation}\label{cvx-rd}
r_Q:=\inf \{r_x:{x\in Q}\}>0;
\end{equation}
see \cite[Theorem 5.3, p. 169]{Sakai1996} or \cite[Lemma 3.1]{LiLi2009}.


Consider now an extended real-valued function $f:M\rightarrow\overline{\IR}:=(-\infty,\infty]$ and let $\mathcal{D}(f)$ denote its domain, that is, $\mathcal{D}(f):=\{x\in M:f(x)\neq\infty\}$. Write  $\Gamma^f_{xy}:=\Gamma^{\mathcal{D}(f)}_{xy}$ for simplicity, that is  $\Gamma^f_{xy}$
 stands for the subset   consisting of all $\gamma_{xy}\in \Gamma_{xy}$ such that $\gamma([0,1]])\subseteq \mathcal{D}(f)$.  In the
following definition, we introduce the notions of the convexity for functions,
where item (c) is known in \cite{LiMWY2011,LiY2012}.

\begin{definition}\label{definition-convex-f}
Let $f:M\rightarrow\overline{\IR}$ be a proper function with a weakly convex domain $\mathcal{D}(f)$,
and let $x\in \mathcal{D}(f)$. Then, $f$ is said to be

{\rm (a)} convex (resp. strictly convex) at $x$ if,
for any $y\in   \mathcal{D}(f)\setminus\{x\}$ and any geodesic $\gamma_{xy}\in
\Gamma^f_{xy}$ the composition $f\circ\gamma_{xy}:[0,1]\rightarrow\IR$ is convex (resp. strictly convex) on $(0,1)$:
\begin{equation}\label{D-convex}
f\circ\gamma_{xy}(t)\le(resp.\;<)(1-t)f(x)+tf(y)\;\;\mbox{ {\rm for all}
$t\in(0,1)$};
\end{equation}

{\rm (b)} weakly convex (resp. weakly strictly convex) at $x$ if, for any $y\in \mathcal{D}(f)$ there exists   $\gamma_{xy}\in {\rm min-}\Gamma^f_{xy}$ such that \eqref{D-convex} holds;

{\rm (c)} 
%
weakly convex (resp. convex, strictly convex, weakly strictly convex) if so is it  at each  $x\in \mathcal{D}(f)$.
\end{definition}

Clearly, for a proper function  $f$ on $M$, 
the convexity   implies the weak convexity, and
 the  strict convexity implies the  convexity.



Let $f:M\rightarrow\overline{\IR}$ be   proper and weakly convex at $x\in \mathcal{D}(f)$. The directional derivative in direction $u\in T_xM$ and the subdifferential of $f$ at $x$ are, respectively, defined by
$$f'(x;u):=\lim_{t\rightarrow0^+}\frac{f(\exp_xtu)-f(x)}{t}$$
and
$$\partial f(x):=\{v\in T_{x}M:\; \langle v,u\rangle\le f'(x;u)\;\;\mbox{for any } u\in T_xM\}.$$
Then, by \cite[Proposition 3.8(iii)]{LiMWY2011},  the following  relationship holds between $\partial f(x)$ and ${\rm cl}f'(x;\cdot)$,
 the lower semi-continuous hull of   $f'(x;\cdot)$:
\begin{equation}\label{r-s-l}
{\rm cl}f'(x;u)=\sup\{\langle u,v\rangle:\; v\in \partial f(x)  \}\quad\mbox{for any }u\in T_xM.
\end{equation}

\begin{lemma}\label{subdifferential representation} Let $f:M\rightarrow\overline{\IR}$ be  proper with a weakly convex domain $\mathcal{D}(f)$. Let  $x\in \mathcal{D}(f)$ and   $v\in T_xM$.

{\rm (i)} If $f$ is weakly convex (resp.  weakly strictly convex) at $x$, then  $v\in\partial f(x)$ if and only if,   for  some   or any constant $r>0$, and for any $y\in\mathcal{D}(f)\cap{\bf B}(x,r)$, there exists a geodesic $\gamma_{xy}\in{\rm min-}\Gamma^f_{xy}$ such that
\begin{equation}\label{subd-ine}
f(y)\ge(resp. >) f(x)+\langle v,\dot{\gamma}_{xy}(0)\rangle.
\end{equation}

{\rm (ii)} If $f$ is  convex (resp.   strictly convex) at $x$, then  $v\in\partial f(x)$ if and only if,   for  some   or any constant $r>0$, 
the inequality \eqref{subd-ine} holds  for any $y\in\mathcal{D}(f)\cap{\bf B}(x,r)$ and any $\gamma_{xy}\in{\rm min-}\Gamma^f_{xy}$.
\end{lemma}

\begin{proof} We only prove assertion (i) (as the proof for assertion (ii) is similar). To do this,
suppose that $f$ is weakly convex (resp.  weakly strictly convex) at $x$. It suffices to verify that  the following statements are equivalent:

\begin{enumerate}[{\rm (a)}]
  \item $v\in\partial f(x)$.
  \item For any $r>0$ and any $y\in\mathcal{D}(f)\cap{\bf B}(x,r)$, there exists a geodesic $\gamma_{xy}\in{\rm min-}\Gamma^f_{xy}$ such that  \eqref{subd-ine} holds.

  \item  There is some $r>0$ such that for any $y\in\mathcal{D}(f)\cap{\bf B}(x,r)$, there exists a geodesic $\gamma_{xy}\in{\rm min-}\Gamma^f_{xy}$ such that \eqref{subd-ine} holds.
\end{enumerate}
  We shall complete the proof by showing the implications
(a) $\Rightarrow$ (b) $\Rightarrow$ (c) $\Rightarrow$ (a).
To do this, assume (a). Then   $v\in\partial f(x)$, and by definition, we have that
\begin{equation}\label{sub-represen1}
\langle v,u\rangle\le f'(x;u)\quad\mbox{for any } u\in T_xM.
\end{equation}
Let $r>0$ and $y\in\mathcal{D}(f)\cap{\bf B}(x,r)$ be arbitrary. Noting that $f$ is weak convex (resp.  weakly strictly convex)  at $x$,  there exists   $\gamma_{xy}\in{\rm min-}\Gamma_{xy}^f$ such that the composite $f\circ\gamma_{xy}:[0,1]\rightarrow\IR$ is convex (resp. strictly convex) on $[0,1]$.
Therefore,
$$
f'(x;\dot{\gamma}_{xy}(0))=\inf_{t>0}\frac{f(\exp_xt\dot{\gamma}_{xy}(0))-f(x)}{t}\le(resp. <) f(y)-f(x).
$$
This, together with \eqref{sub-represen1},  yields that
$$\langle v,\dot{\gamma}_{xy}(0)\rangle\le f'(x;\dot{\gamma}_{xy}(0))\le(resp. <) f(y)-f(x).$$
Hence (b) holds, and the implication (a) $\Rightarrow$ (b) is checked.
Noting  that the implication (b) $\Rightarrow$ (c)  is evident,  it remains to  show the implication (c) $\Rightarrow$ (a). To this end,
assume (c). Then  one can  choose   $r>0$  and  $\gamma_{xy}\in{\rm min-}\Gamma^f_{xy}$ for any $y\in\mathcal{D}(f)\cap{\bf B}(x,r)$ such that \eqref{subd-ine} holds. Without loss of generality, one could assume that $r\le r_x$.
Let $u\in T_xM\setminus\{0\}$ be arbitrary, and set  $s_0:= \frac{ r }{\|u\|}$. Then,
 for any $s\in(0,s_0)$, $y(s):=\exp_x(su)\in{\bf B}(x,r)\subseteq {\bf B}(x,r_x)$. It follows that  the geodesic $\gamma_{xy(s)}$  joining $x$ and $y(s)$ is
 unique. Therefore, $\gamma_{xy(s)}$ is of the form:
$$\gamma_{xy(s)}(t)=\exp_x( t(su))\quad\mbox{for each }t\in [0,1],$$
and if $y(s)\in\mathcal{D}(f)$, then \eqref{subd-ine} holds with $y(s)$ in place of $y$:
$$\langle v,su\rangle=\langle v,\dot\gamma_{xy(s)}(0)\rangle\le  f(y(s))-f(x)= f(\exp_x(su))-f(x)\quad\mbox{for any }s\in(0,s_0).$$
Note that the above inequality holds trivially if $y(s)\notin\mathcal{D}(f)$. Then, by definition, we get that
$\langle v,u\rangle\le f'(x;u)$, and so  (a) holds as $u\in T_xM$ is arbitrary. Thus, the implication (c) $\Rightarrow$ (a) is shown and the proof is complete.
\end{proof}

%

Fix $\bar x\in {\cal D}(f)$ and recall that $f$ is center Lipschitz continuous at $\bar x$
if there exits a neighborhood $U$ of $\bar x$ and  a constant $L$ such that
$$|f(x)-f(\bar x)|\le L{\rm d}(x,\bar x)\quad\mbox{for any } x\in U.$$
The  center  Lipschitz constant $L^f_{\bar x}$ at $\bar x$ is defined to be the minimum of all $L$ such that above inequality holds for some neighborhood $U$ of $\bar x$. Then it is clear that
$$
 L^f_{\bar x}=\lim_{\delta\to 0+}\sup\left\{\frac{|f(x)-f(\bar x)|}{{\rm d}(x,\bar x)}:\;  0<{\rm d}(x,\bar x)\le \delta\right\}.
$$

The following properties about the subdifferential of a (weakly) convex function can be found in \cite[Proposition 6.2]{LiY2012} except \eqref{bound-subd}
by definition.

{
\begin{lemma}\label{Convex-f-P}
Let $f:M\rightarrow\overline{\IR}$ be a proper function.
Then the following assertions hold:

{\rm (i)} If is weakly convex , then $f$ is continuous on ${\rm int}\mathcal{D}(f)$.

{\rm (ii)} If  $\bar x\in {\rm int}\mathcal{D}(f)$ 
%
and   $f$ is weakly convex  at $\bar x$, then $\partial f(\bar x)$ is a nonempty, compact and convex set satisfying
\begin{equation}\label{bound-subd}
\|v\|\le L_{\bar x}^f\quad\mbox{for any }v\in \partial f(\bar x).
\end{equation}

\end{lemma}
}

%
The following lemma, which provides some sufficient conditions ensuring the sum rule of subdifferential, was proved in  \cite[Proposition 4.3]{LiMWY2011}.

\begin{lemma}\label{Sub-Sum} Let
  $f,g:M\rightarrow\overline{\IR}$ be proper functions such that  $f,g$ and $f+g$ are weakly convex
    at $x\in {\rm int}\mathcal{D}(f)\cap\mathcal{D}(g)$. Then the following sum rule for the subdifferential holds:
$$\partial(f+g)(x)=\partial f(x)+\partial g(x).$$
\end{lemma}


%

\subsection{VIP: existence  and convexity properties of solution sets}

Let $Q\subseteq M$ be a nonempty subset and let $A:Q\rightrightarrows{TM}$ be a set-valued vector field defined on $Q$, that is, $A(x)\subseteq T_xM$ is nonempty for each $x\in Q$. Consider the following variational inequality problem (${\rm VIP}$ for short) associated to the pair $(A,Q)$: To find a point $\bar x\in Q$ such that
 \begin{equation}\label{VIP}
\exists \bar v\in A(\bar x) \mbox{ s.t. $\langle\bar v,\dot{\gamma}_{\bar xy}(0)\rangle\ge0$ for any $y\in Q$ and $\gamma_{\bar xy}\in\Gamma_{\bar xy}^Q$}.
\end{equation}
Any point $\bar x\in Q$ satisfying \eqref{VIP} is called a solution of ${\rm VIP}$, and the set of all solutions is denoted by ${\rm VIP}(A,Q)$.

{Variational inequality problem \eqref{VIP} was first introduced in \cite{Nemeth2003}, for single-valued vector fields on Hadamard manifolds, and extended respectively  in \cite{LiLi2009} and  \cite{LiY2012} for single-valued vector fields and multivalued vector fields on general Riemannian manifolds}. As we have mentioned previously,   our approach to solve the EP is founded strongly on  some existence results about the ${\rm VIP}$, which are taken from \cite{LiY2012}. For this purpose, we recall some notions in the following definition; see, e.g, \cite{Li2009,LiY2012}.

\begin{definition}
Let $Q\subseteq M$ be a subset and $A:Q\rightrightarrows{TM}$ be a set-valued vector field on $Q$.  $A$ is said to be

{\rm (a)}  upper semi-continuous (usc for short) at $x_0$, if, for any open set $U$ satisfying $A(x_0)\subseteq U\subseteq T_{x_0}M$,
there exists an open neighborhood $U(x_0)$ of $x_0$ such that $P_{x_0,x}A(x)\subseteq U$ for any $x\in U(x_0)\cap Q$;

{\rm (b)} upper Kuratowski semi-continuous (uKsc for short) at $x_0$ if, for any sequences $\{x_k\}\subset Q$ and $\{u_k\}\subset TM$ with each $u_k\in A(x_k)$, relations $\lim_{k\rightarrow\infty}x_k=x_0\in Q$ and $\lim_{k\rightarrow\infty}u_k=u_0$ imply $u_0\in A(x_0)$;

{\rm (c)} usc (resp., uKsc) on $Q$ if it is usc (resp., uKsc) at each $x\in Q$.
\end{definition}

By definition, it is evident that upper semi-continuity implies upper Kuratowski semi-continuity.
In the following example, we provide two set-valued vector fields which are useful for our study in next sections  of  the present paper.

\begin{example}\label{remak-p}
Fix a point  $y\in M$, and define vector fields $ {\rm Exp}_{(\cdot)}^{-1}y: \rightrightarrows{TM}$ and $\exp_{(\cdot)}^{-1}y: \rightrightarrows{TM}$ respectively by
$$
 {\rm Exp}_x^{-1}y:=\{u\in T_xM:\exp_xu=y\}\quad\mbox{for each }x\in M,
$$
and
\begin{equation}\label{EYX}
\exp_{x}^{-1}y:=\{u\in   {\rm Exp}_x^{-1}y:\|u\|={\rm d}(x,y)  \}\quad\mbox{for each }x\in M.
\end{equation}
Then one can check easily by definition and Lemma \ref{cvg-exp} that
  ${\rm Exp}_{(\cdot)}^{-1}y$ is uKsc on $M$
and $\exp_{(\cdot)}^{-1}y$ is usc on $M$.
\end{example}

Recall from \cite{LiY2012} that a point $o\in Q$ is called a weak pole of Q if for each $x\in Q$, ${\rm min-}\Gamma_{ox}$ is a singleton and ${\rm min-}\Gamma_{ox}\subseteq Q$. Clearly, any subset with a weak pole is connected. The notions of the monotonicity in
the following definition are well known; see for example \cite{Colao2012,LiY2012}.

\begin{definition}\label{monotone-svf}
Let $Q\subseteq M$ be a 
subset and $A:Q\rightrightarrows{TM}$ be a set-valued vector field. The vector field $A$ is said to be

{\rm (a)} monotone on $Q$ if, for any $x,y\in Q$ and $\gamma_{xy}\in\Gamma_{xy}^Q$ the following inequality
holds:
$$
\langle v_x,\dot{\gamma}_{xy}(0)\rangle-\langle v_y,\dot{\gamma}_{xy}(1)\rangle\le 0\;\;\;\;\mbox{for any } v_x\in A(x),\;v_y\in A(y);
$$

{\rm (b)} strictly monotone on $Q$ if it is monotone and,
for any $x,y\in Q$ with $x\neq y$ and $\gamma_{xy}\in\Gamma_{xy}^Q$ the following inequality
holds:
$$
\langle v_x,\dot{\gamma}_{xy}(0)\rangle-\langle v_y,\dot{\gamma}_{xy}(1)\rangle< 0\;\;\;\;\mbox{for any } v_x\in A(x),\;v_y\in A(y).
$$
\end{definition}

Let $Q\subseteq M$ be a closed connected and locally convex set. By \cite[p. 170]{Sakai1996}, there exists a connected (embedded) $k$-dimensional totally geodesic sub-manifold $N$ of $M$ such that $Q=\overline{N}$. Following \cite{LiY2012}, the set ${\rm int}_RQ:=N$ is called the relative interior of $Q$. Moreover, as in \cite{LiY2012}, we say that a closed locally convex set $Q$ has the BCC (bounded convex cover) property if   there exists $o\in Q$ such that, for any $R\ge0$, there exists a weakly convex compact subset of $M$ containing $Q\cap {\bf B}(o,R)$.

\begin{remark}
We remark that the notion of the BCC property defined above is a litter stronger than that defined in \cite[Definition 3.9]{LiY2012}, where it is required that the compact subset containing $Q\cap {\bf B}(o,R)$ is ``locally convex"  rather than  ``weakly convex". From its proof, one sees  that the BCC property assumption defined in \cite[Definition 3.9]{LiY2012} seems  insufficient  for \cite[Theorem 3.10]{LiY2012}, while
 the   stronger version of the  BCC property   defined here  is  sufficient.
\end{remark}



%
%
For the remainder,  we  use $\mathcal{V}(Q)$ to denote the set of all uKsc set-valued vector
fields $A$ such that $A(x)$ is compact and convex for each $x\in Q$.

%
%
%
%
%
%
 Proposition \ref{LY-Existence-coercive} below   extends the corresponding existence result    in \cite[Theorem 3.10]{LiY2012} (see the explanation made in Remark \ref{remark-co}).
The proof of   Proposition \ref{LY-Existence-coercive} is similar to that for  \cite[Theorem 3.10]{LiY2012}, and is kept here for completeness.

\begin{proposition}\label{LY-Existence-coercive} Let $Q\subseteq M$ be a closed locally convex subset with a weak pole $o\in{\rm int}_RQ$ and $A\in\mathcal{V}(Q)$. Then ${\rm VIP}(A,Q)\neq\emptyset$ provided one of the following assumptions holds: 

{\rm (a)} $Q$ is compact; 


  {\rm (b)} $Q$ has the BCC property and
      there exists  a compact subset $L\subseteq M$ such that
       \begin{equation}\label{L_compact}
      \begin{array}{l}
        x\in Q\setminus L\Rightarrow
        [\forall v\in A(x),\exists y\in Q\cap L,\,\gamma_{xy}\in{\rm min-}\Gamma_{xy}^Q \mbox{ s.t.  }\langle v,\dot{\gamma}_{xy}(0)\rangle<0].
      \end{array}
      \end{equation}
\end{proposition}

\begin{proof} It was known in \cite[Theorem 3.6]{LiY2012} in the case when $Q$ is compact. Below we assume that assumption (b) holds.
Then,  there exists a compact subset
   $L$ 
   such that \eqref{L_compact} holds. Then
     there exist $R>0$ and a weakly convex and  compact subset $K_R$ of $M$ such that
       $ L\subset   {{\bf B}(o,R)}$ and $Q\cap  {\bf B}(o,R)\subseteq K_R$. 
Write $Q_R:=Q\cap \overline{{\bf B}(o,R)}$, and $\hat Q_R:=Q\cap K_R$ for saving the print space. Then
$$ Q\cap L\subseteq Q\cap  {\bf B}(o,R)\subseteq Q_R\subseteq \hat Q_R.
$$
Thus, by \eqref{L_compact}, one checks that
\begin{equation}\label{LYPL-1}
{\rm VIP}(A,\hat Q_R)\subseteq{\rm VIP}(A, Q_R)\subseteq Q\cap L\subseteq {{\bf B}(o,R)}.
\end{equation}
 Moreover, since $o\in{\rm int}_RQ$ is a weak pole of $Q$ (and so the minimal geodesic $\gamma_{ox}$ joining $o$ to $x$ is unique) and  $K_R$ is weakly convex, one can check by definition that
 $o$ is a weak pole of $\hat Q_R$ and  $o\in{\rm int}_R\hat Q_R$ (noting that $Q\cap  {\bf B}(o,R)\subseteq K_R$).
 Thus \cite[Theorem 3.6]{LiY2012} is applied (with $\hat Q_R$ in place of $A$) to get that ${\rm VIP}(A,\hat Q_R)\neq\emptyset$.
 In view of \eqref{LYPL-1},  $\emptyset\neq{\rm VIP}(A,Q_R)\subseteq {{\bf B}(o,R)}$, and
it follows from \cite[Proposition 3.2]{LiY2012} that
$$ {\rm VIP}(A,Q_R)={\rm VIP}(A,Q_R)\cap{{\bf B}(o,R)}\subseteq{\rm VIP}(A,Q),$$
and so   ${\rm VIP}(A,Q)\neq\emptyset$,  completing the
proof.
%
\end{proof}

\begin{remark}\label{remark-co}
Let $Q\subseteq M$ be a locally convex subset with a weak pole $o\in{\rm int}_RQ$. Recall from \cite{LiY2012} that the vector field
  $A$ satisfies the coerciveness condition on $Q$ if
$$
 \sup_{v_{o}\in A(o),v_x\in A(x)}\frac{\langle v_x,\dot{\gamma}_{xo}(0)\rangle-\langle v_{o},\dot{\gamma}_{xo}(1)\rangle}{{\rm d}(o,x)}\rightarrow-\infty\quad\mbox{as }{\rm d}(o,x)\rightarrow +\infty\mbox{ for }x\in Q.
$$
Then one   checks directly by definition  that the coerciveness condition for $A$ implies   that there exists  a compact subset $L\subseteq M$ such that  \eqref{L_compact} in   {\rm (b)} of Proposition \ref{LY-Existence-coercive}  holds {(noting that $A(o)$ is compact). However, the converse is not true, in general, even in the Euclidean space setting. To see this, one may consider the simple mapping  $A$ on $Q:=\IR$ defined by
 $A(x):=[-1,1]$ if $x=0$ and  $A(x):={\rm sign} (x)$ otherwise. }
Thus Proposition \ref{LY-Existence-coercive} is an extension of the corresponding existence result    in \cite[Theorem 3.10]{LiY2012}.
\end{remark}

As usual, we set $D_{\kappa}:=\frac{\pi}{\sqrt{\kappa}}$ if $\kappa>0$ and $D_{\kappa}:=+\infty$ if $\kappa\le 0$  (see e.g., \cite{LiY2012,Sakai1996}).
The following proposition lists some   results on the structure  of    the solution set ${\rm VIP}(A,Q)$, 
which are known in  \cite[Theorems 3.13, 4.6 and 4.8]{LiY2012}, respectively.

\begin{proposition}\label{LY-unique} Suppose that $A\in\mathcal{V}(Q)$ is   monotone on $Q\subseteq M$ and ${\rm VIP}(A,Q)\not=\emptyset$.
Then the following assertion holds:

{\rm (i)} If $Q$ is a locally convex subset, then the solution set ${\rm VIP}(A,Q)$ is locally convex.

{\rm (ii)} If 
$A$ is strictly monotone on $Q$, then  ${\rm VIP}(A,Q)$ is a singleton.

 {\rm (iii)} If $M$ is of the sectional curvatures bounded from above by some   $\kappa\in [0,+\infty)$ and $Q$ is a $D_{\kappa}$-convex subset, then the solution set ${\rm VIP}(A,Q)$ is $D_{\kappa}$-convex.
. 
\end{proposition}



%
%
%
%

\section{Equilibrium problem}
Throughout the whole section, we always assume that

$\bullet$ $Q\subseteq M$ is  a nonempty closed and locally convex subset;

$\bullet$ $F:M\times M\rightarrow\overline \IR$ 
is a proper bifunction with $0\le F(x,x)<+\infty$ for any $x\in Q$.\\
The domain $\mathcal{D}(F)$ of $F$ is defined by $$\mathcal{D}(F):=\{(x,y)\in M\times M:-\infty<F(x,y)<+\infty\}.$$
Recall that the {\rm EP} associated to the pair $(F,Q)$ 
is to find a point $\bar x\in Q$ such that $F(\bar x,y)\ge0$ for any $y\in Q$. 
Any point $\bar x\in Q$ satisfying \eqref{EP} is called a solution of ${\rm EP}$ \eqref{EP}, and the set of all solutions is denoted by ${\rm EP}(F,Q)$.

\subsection{Properties of bifunctions}

In the following definition we introduce some monotonicity and convexity notions for bifuctions on Riemannian manifolds. In particular,  the
corresponding ones of items (a) and (b) in linear spaces are refereed to, for example, \cite{Chadli2000,Konnov2009}; while item (c) as far as we know are new and  plays a key role  in our study in the present paper.

\begin{definition}\label{monotone-bif}
The bifunction $F$ is said to be 

{\rm (a)} monotone on $Q\times Q$ if $F(x,y)+F(y,x)\le0$ for any $(x,y)\in Q\times Q$;

{\rm (b)} strictly monotone on $Q\times Q$ if  $F(x,y)+F(y,x)<0$ for any $(x,y)\in Q\times Q$  with  $ x\neq y$ and
\begin{equation}\label{f00}
 F(x,x)=0 \quad\mbox{for  any }x\in Q;
\end{equation}

{\rm (c)} point-wise weakly convex (resp. point-wise convex) on Q if,
for any   $x\in Q$, the function $F(x,\cdot):M\rightarrow\overline{\IR}$ is weakly convex (resp. convex) at $x$.
\end{definition}

Note that  if $F$ is monotone on $Q\times Q$, then \eqref{f00} holds (as  $F(x,x)\ge0$ for any $x\in Q$ by assumption).


Let $V:Q\rightrightarrows {TM}$ be a vector field. 
   Associated to $V$, we
 %
%
define  the bifunction $G_V:M\times M\rightarrow \overline{\IR}$   by
\begin{equation}\label{f-G-G}
G_V(x,y):=\sup_{u\in V(x),v\in \exp_{x}^{-1}y}\langle u,v\rangle\;\;\;\;\mbox{for any } (x,y)\in M\times M,
\end{equation}
where   for any $(x,y)\in M\times M$,  $\exp_{x}^{-1}y$ is  defined by \eqref{EYX}, and  we adopt the
the convention that $\sup\emptyset=+\infty$.
Proposition \ref{example} below provides some properties of the bifunctions $G_V$ that will be used in the sequel.
As usual, for a subset $Z$ of $T_xM$, we use  $\overline{{\rm co}}Z$ to  denotes the closed and convex hull of the set $Z$ in $T_xM$.

\begin{proposition}\label{example} Suppose $V(x)\subseteq T_xM$ is   nonempty   for each $x\in Q$, and let $G_V$ be defined by \eqref{f-G-G}.  Then the following assertions hold:

{\rm (i)} If $V(x)$ is compact-valued, then
\begin{equation}\label{domain-G-G}
\mathcal{D}(G_V)=Q\times M\quad\mbox{and}\quad G_V(x,x)=0\;\;\mbox{ for any }x\in Q.
\end{equation}

{\rm (ii)} $G_V(x,\cdot)\circ\gamma_{xy}$ is convex on $[0,1]$ for any $x,\,y\in Q$ and any geodesic $\gamma_{xy}\in{\rm min-}\Gamma_{xy}$.

{\rm (iii)} If $G:Q\times Q\to \overline\IR$ is  point-wise weakly convex on $Q$, then so is   $G_V+G$.

{\rm (iv)} $\partial G_V(x,\cdot)(x)={\overline {\rm co}}V(x)$  for any $x\in Q$.

{\rm (v)} If $V(x)$ is compact-valued   and $V$   usc on $Q$, then  the function  $x\mapsto G_V(x,y)$ is usc on $Q$ for each $y\in Q$.

 %
%
%
%
\end{proposition}

\begin{proof}
Assertion (i) is clear by definition. 
To show assertion (ii), fix $x, y\in Q$ and let $\gamma_{xy}\in {\rm min-}\Gamma_{xy}$ and $y_t:=\gamma_{xy}(t)$.  Then we have that
\begin{equation}\label{sbet}
  \exp_{x}^{-1}{y_t}\subseteq t\, \exp_{x}^{-1}{y}\quad\mbox{for each }t\in(0,1).
\end{equation}
  Indeed,  let $v_t\in \exp_{x}^{-1}{y_t}$ with some $t\in (0,1)$. Then, $\|v_t\|={\rm d}(x,y_t)=t{\rm d}(x,y)$ and $\exp_xv_t=y_t$.
 Define a curve $\beta:[0,1]\rightarrow M$ by
\begin{equation}
\beta(s):=\left\{\begin{array}{ll}
\exp_x\frac{s}{t}v_t,\quad&\mbox{$s\in[0,t]$},\\
\gamma_{xy}(s),&\mbox{$s\in(t,1]$}.
\end{array}
\right.
\nonumber
\end{equation}
Then  $ l(\beta)=\|v_t\|+{\rm d}(y_t,y)={\rm d}(x,y)$. This means that $\beta\in{\rm min-}\Gamma_{xy}$
thanks to Lemma \ref{minimal-geodesic}. Therefore, $ \frac{1}{t}v_t\in \exp_{x}^{-1}{y}$ by definition because  $\dot{\beta}(0)=\frac{1}{t}v_t$;
hence \eqref{sbet} holds.  Thus
 $$G_V(x,\gamma_{xy}(t))
 = \sup_{u\in V(x),v_t\in \exp_{x}^{-1}{y_t}}\langle u, {v_t} \rangle
 \le t\sup_{u\in V(x),v\in \exp_{x}^{-1}{y}}\langle u,v\rangle=tG(x,y).  $$
 This shows that  $G_V(x,\gamma_{xy}(\cdot))$ is convex on $[0,1]$ (noting that $G_V(x,x)=0$), and assertion (ii) is shown as 
  $\gamma_{xy}\in {\rm min-}\Gamma_{xy}$ is arbitrary.

Assertion (iii) follows immediately from assertion (ii). Now, we verify assertion (iv). To proceed, let $x\in Q$
and 
$\xi\in T_xM$. Then  for any  $t>0$   small enough, one has that $\exp_{x}^{-1}{{\exp_xt\xi}}=\{t\xi\}$. Thus noting that $G_V(x,x)=0$, we have  by definition
that
\begin{equation}\label{GSub-Con-1}
G_V(x,\cdot)'(x;\xi)
=\lim_{t\rightarrow0^+}\frac{\sup_{u\in V(x)}\langle u,t\xi\rangle}{t}=\sup_{u\in V(x)}\langle u,\xi\rangle=\sup_{u\in{\overline {\rm co}}{V(x)}}\langle u,\xi\rangle.
\end{equation}
This, together with  \eqref{r-s-l},   implies that
$$\sup_{u\in{\overline {\rm co}}{V(x)}}\langle u,\xi\rangle\ge \sup_{u\in\partial G_V(x,\cdot)(x)}\langle u,\xi\rangle\quad\mbox{for any }\xi\in T_xM,
$$
and so $\partial G_V(x,\cdot)(x)\subseteq {\overline {\rm co}}{V(x)}$ by \cite[Corollary 13.1.1, p.113]{Rockafellar1970}.
Moreover, by \eqref{GSub-Con-1}, one have by definition that
$$
\partial G_V(x,\cdot)(x)=\partial G_V(x,\cdot)'(x;0)  \supseteq {\overline {\rm co}}{V(x)}.
$$
 Thus assertion (iv) is shown.

It remains to show assertion (v). To this end, fix $y\in Q$. Let $\varepsilon >0$, $x\in Q$ and let $\{x_n\}\subseteq Q$ be such that $\lim_{n\rightarrow\infty}x_n=x$.
Since $V$ and $\exp_{(\cdot)}^{-1}{y}$ are usc at $x$ (see Lemma \ref{remak-p}),  there is $K\in\IN$ such that
\begin{equation}\label{upper-cot-P3}
P_{x,x_n}V(x_n)\subseteq{\bf B}(V(x),\varepsilon)\quad\mbox{and}\quad P_{x,x_n}\exp_{x_n}^{-1}{y}\subseteq{\bf B}(\exp_{x}^{-1}{y},\varepsilon)\quad \mbox{for each }n\ge K.
\end{equation}
Set $R:=\max\{|V(x)|,{\rm d}(x,y)\}$ (where $|V(x)|:=\max_{v\in V(x)}\{\|v\|\}<+\infty$ as $V(x)$ is compact), and, without loss of generality, assume that $\varepsilon<R$. Then, it follows from \eqref{upper-cot-P3} that, for any $n>K$,
$$\sup_{v_n\in V(x_n), u_n\in \exp_{x_n}^{-1}{y}}\langle v_n,u_n\rangle\le\sup_{ v\in {\bf B}(V(x),\varepsilon), u\in {\bf B}(\exp_{x}^{-1}{y},\varepsilon)}\langle  v, u\rangle\le\sup_{v\in V(x), u\in \exp_{x}^{-1}{y}}\langle v,u\rangle+3\varepsilon R,$$
and so 
$\overline{\lim}_{n\rightarrow\infty}G_V(x_n,y)\le G_V(x,y)+ 3R\varepsilon$. Thus, assertion (v) holds as $\varepsilon>0$ is arbitrary  and
the proof is complete.
 %
%
%
 \end{proof}


\subsection{Relationship between VIP and EP}

 For the remainder of the paper, we  will  make use of the following hypotheses for the bifunction $F$, where, as usual, we
  use $\delta_{C}(\cdot)$ to denote the indicator function of the nonempty subset $C$ 
  defined  by $\delta_C(x):=
0$   if  $ x\in C$  and  $+\infty$  otherwise:


{\rm (H1)}  $F$ is point-wise weakly convex on $Q$ and,  $x\in {\rm int} \mathcal{D}(F(x,\cdot))$ for each $x\in Q$.

{\rm (H2)}   $F+\delta_{Q\times Q}$ is point-wise weakly convex on $Q$.

{\rm (H3)} For any $y\in Q$, the function $x\mapsto F(x,y)$ is usc  on $Q$.

{\rm (H4)} The function $x\mapsto F(x,x)$ is lower semi-continuous (lsc for short) on $Q$.

{
\begin{remark}
We remark that 
the latter part of hypothesis {\rm (H1)} is particularly satisfied if $Q\times Q\subseteq{\rm int}\mathcal{D}(F)$. The first part of hypothesis {\rm (H1)} and hypothesis {\rm (H2)} 
are satisfied in the case when $Q$ is weakly convex and $F(x,\cdot)$ is (weakly) convex for any $x\in Q$, which, together with Hypothesis {\rm (H3)} are     standard assumption for the {\rm EP} (see, e.g, \cite{Chadli2000,Colao2012,Combettes2005,Iusem2009,Batista2015}); while
   hypothesis {\rm (H4)} is particularly satisfied if  $F(x,x)=0$ for any $x\in Q$ (which was used in \cite{Chadli2000,Combettes2005,Iusem2009}).
\end{remark}
}

Note that, by definition, the following implication holds:
\begin{equation}\label{A-Weak-C}
\mbox{{\rm (H2)}$\Longrightarrow Q$ is weakly convex}.
\end{equation}
Associated to the pair $(F,Q)$, we define  the set-valued vector field  $A_F:Q\rightrightarrows {TM}$   by
\begin{equation}\label{VIP-10}
A_F(x):=\partial F(x,\cdot)(x)\quad \mbox{for any $x\in Q$}.
\end{equation}
Then the  following proposition is clear from Lemma \ref{Convex-f-P} (ii). 
\begin{proposition}\label{P-PWF-0}
Suppose that $F$ satisfies (H1). Then, 
the set-valued vector field  $A_F$  is well-defined,  compact convex-valued on $Q$  and  
satisfies
  $$\max_{v\in A_F(x)}\|v\|\le L_x^F \quad\mbox{for each }x\in Q,
$$
where $L_x^F$ stands for  the center  Lipschitz constant of $F(x,\cdot)$ at $x$.
\end{proposition}

The following proposition establishes the relationship between the {\rm EP}   associated to  the pair $(F,Q)$  and the {\rm VIP} associated  to the pair  $(A_F, Q)$. 







\begin{proposition}\label{relation-EP-VIP}
Suppose that $F$ satisfies {\rm (H1)} and {\rm (H2)}. Then
\begin{equation}\label{VIPinEP}
{\rm VIP}(A_F,Q)\subseteq{\rm EP}(F,Q),
\end{equation}
and the equality holds if \eqref{f00} is additionally assumed.
\end{proposition}

\begin{proof}
Let $\bar x\in Q$ and note that $F$ satisfies {\rm (H1)} and {\rm (H2)}. Then, by implication \eqref{A-Weak-C}, $Q$ is weakly convex, and then the same argument  for proving
  \cite[Proposition 6.4]{LiY2012} (with $F(\bar x,\cdot),\,Q$ in place of $f,\,A$ there) works for the following   equivalence:
\begin{equation}\label{eqvip}
 \bar x\in {\rm VIP}(A_F,Q)
 \Longleftrightarrow [F(\bar x,y)\ge F(\bar x,\bar x) \; \mbox{ for any } y\in Q].
\end{equation}
 Thus \eqref{VIPinEP} follows from   the assumption that $F(\bar x,\bar x)\ge0$;   while the converse inclusion of  \eqref{VIPinEP} holds trivially by \eqref{eqvip} if \eqref{f00} is additionally assumed. The proof is complete. 
%
\end{proof}

\begin{proposition}\label{P-PWF} Suppose that $F$ satisfies {\rm (H1)}. Then the following assertions hold:

 {\rm (i)} If 
 $F$ is  monotone (resp. strictly monotone)  on $Q\times Q$, then so is  $A_F$    on $Q$.

{\rm (ii)} If $F$ satisfies   {\rm (H3)} and {\rm (H4)},   then
$A_F$ is uKsc on $Q$; hence $A_F\in\mathcal{V}(Q)$.
\end{proposition}

\begin{proof}
(i). Suppose that $F$ is monotone on $Q\times Q$.
Let $x,y\in Q$, $u_x\in A_F(x), u_y\in A_F(y)$ and let $\gamma_{xy}\in \Gamma_{xy}^Q$. We have to show
\begin{equation}\label{monotone-formular}
  \langle u_x,\dot{\gamma}_{xy}(0)\rangle-\langle u_y,\dot{\gamma}_{xy}(1)\rangle\le0.
\end{equation}
To do this, subdivide $\gamma_{xy}$ into $n$ subsegments with the equal length determined by the consecutive points
$$
x=x_0<x_1<\ldots<x_{n-1}<x_n=y
$$
such that
$$
{\rm d}(x_{i-1},x_i)=\frac{l(\gamma_{xy})}{n}\le \bar r.\quad i=1,2,\dots,n,
$$
where $\bar r:=\min\{r_z: z\in {\gamma_{xy}[0,1]}\}>0$  by \eqref{cvx-rd}.
Thus, for each $i=1,2,\dots,n$, $\exp_{x_{i-1}}^{-1}x_i$ is a singleton, and
\begin{equation}\label{gamma-x-i}
 {\rm min-}\Gamma_{x_{i-1}x_i}=\{\gamma_{x_{i-1}x_i}\} \quad\mbox{with} \quad
 \gamma_{x_{i-1}x_i}(\cdot):=\exp_{x_{i-1}}(\cdot)(\exp_{x_{i-1}}^{-1}x_i) .
\end{equation}
 Moreover, we have that
\begin{equation}\label{EP-VIP-monotone20}
\exp_{x_{0}}^{-1}x_1=\frac{1}{n}\dot{\gamma}_{xy}(0),\quad  \exp_{x_{n}}^{-1}x_{n-1}=-\frac{1}{n}\dot{\gamma}_{xy}(1),
\end{equation}
and
\begin{equation}\label{monotone-bf0}\exp_{x_{i}}^{-1}x_{i+1}+\exp_{x_{i}}^{-1}x_{i-1}=0\quad\mbox{for each }i=1,2,\dots,n-1.
\end{equation}
To proceed, set $u_0:=u_x$, $u_n:=u_y$ and  take  $u_i\in A_F(x_i)$ for each $i=1,2,\ldots, n-1$. Now fix $i=1,2,\dots,n$. Then, by assumption (H1),
 Lemma \ref{subdifferential representation} (i) is applicable, and thus,  thanks to \eqref{gamma-x-i}, we have that
$$F(x_{i-1},x_i)\ge \langle u_{i-1},\exp_{x_{i-1}}^{-1}x_{i}\rangle\quad\mbox{and}\quad F(x_i,x_{i-1})\ge \langle u_i,\exp_{x_{i}}^{-1}x_{i-1}\rangle,$$
as $F(x_i,x_i)\ge0$.
This, together with the monotonicity of $F$, implies that
  $\langle u_{i-1},\exp_{x_{i-1}}^{-1}x_{i}\rangle+\langle u_i,\exp_{x_{i}}^{-1}x_{i-1}\rangle\le0 $; hence,
\begin{equation}\label{sum-00}
\sum_{i=1}^n\left(\langle u_{i-1},\exp_{x_{i-1}}^{-1}x_{i}\rangle+\langle u_i,\exp_{x_{i}}^{-1}x_{i-1}\rangle\right)\le0.
\end{equation}
Since by \eqref{monotone-bf0}, $\langle u_{i},\exp_{x_{i}}^{-1}x_{i-1}\rangle+\langle u_{i},\exp_{x_{i}}^{-1}x_{i+1}\rangle )=0 $ for each $i=1,2,\dots,n$, and since
$$
\begin{array}{ll}
 &\langle u_{0},\exp_{x_{0}}^{-1}x_{1}\rangle+\dsum_{i=1}^{n-1}\left(\langle u_{i},\exp_{x_{i}}^{-1}x_{i-1}\rangle+\langle u_{i},\exp_{x_{i}}^{-1}x_{i+1}\rangle\right)+\langle u_{n},\exp_{x_{n}}^{-1}x_{n-1}\rangle\\
 &=\dsum_{i=1}^n\left(\langle u_{i-1},\exp_{x_{i-1}}^{-1}x_{i}\rangle+\langle u_i,\exp_{x_{i}}^{-1}x_{i-1}\rangle\right),
\end{array}
$$
it follows from \eqref{sum-00} that
 $\langle u_{0},\exp_{x_{0}}^{-1}x_{1}\rangle+\langle u_{n},\exp_{x_{n}}^{-1}x_{n-1}\rangle\le0$. Thus \eqref{monotone-formular}
  is seen to hold by    \eqref{EP-VIP-monotone20}, and the proof for assertion (i) is complete. %

(ii). Let $x_0\in Q$ and let $\{x_k\}\subset Q$, $\{u_k\}\subset TM$ with each $u_k\in A_F(x_k)$ such that
\begin{equation}\label{dfgd}
 \lim_{k\rightarrow\infty}x_k=x_0 \quad\mbox{and}\quad \lim_{k\rightarrow\infty}P_{x_0,x_k}u_k=u_0 .
\end{equation}
It suffices  to show $u_0\in A_F(x_0)$. 
To do this, write   $r_{{\bf B}}:=r_{{\bf B}(x_0,r_{x_0})}>0$ (see \eqref{cvx-rd}).   Without loss of generality, we may assume that
 $x_k\in{\bf B}(x_0,\frac{r_{{\bf B}}}{2})$ for all $k$.  Let $y\in {\bf B}(x_0,\frac{r_{{\bf B}}}{2})$. Then, for each $k$,
$${\rm d}(x_k,y)\le {\rm d}(x_0,y)+{\rm d}(x_0,x_k)\le r_{{\bf B}},$$
and so  $\exp_{x_k}^{-1}y$ is a singleton. It immediately follows from \eqref{dfgd} and Lemma \ref{cvg-exp} that
$$
\lim_{k\rightarrow\infty}\langle u_k,\exp_{x_k}^{-1}y\rangle=\langle u_0,\exp_{x_0}^{-1}y\rangle.
$$
Now, suppose hypotheses (H3) and (H4) are satisfied. Then,
\begin{equation}\label{KUC4}
F(x_0,y)\ge{\overline\lim_{k\rightarrow\infty}}F(x_k,y)\quad\mbox{and}\quad F(x_0,x_0)\le \underline{{\lim}}_{k\to\infty} F(x_k,x_k).
\end{equation}
Recalling that each  $u_k\in A_F(x_k)=\partial F(x_k,\cdot)(x_k)$, we get by (H1) that
$$
F(x_k,y)\ge F(x_k,x_k)+\langle u_k,\exp_{x_k}^{-1}y\rangle\quad\mbox{for each }k;
$$
hence
$$
\begin{array}{lll}
{\overline\lim}_{k\to\infty}F(x_k,y)&\ge&\underline{{\lim}}_{k\rightarrow\infty}(F(x_k,x_k)+\lim_{k\to \infty}\langle u_k,\exp_{x_k}^{-1}y\rangle)\\
&=& \underline{{\lim}}_{k\rightarrow\infty}F(x_k,x_k)+\langle u_0,\exp_{x_0}^{-1}y\rangle.
\end{array}
$$
This, together with   \eqref{KUC4}, yields that
$$F(x_0,y)\ge F(x_0,x_0)+\langle u_0,\exp_{x_0}^{-1}y\rangle.$$
Thus, Lemma \ref{subdifferential representation} is applicable to concluding  that $u_0\in \partial F(x_0,\cdot)(x_0)=A_F(x_0)$ as $y\in{\bf B}(x_0,\frac{r_{{\bf B}}}{2})$ is arbitrary, and the upper Kuratowski semi-continuity of  $A_F$ is proved. Furthermore,  by Proposition \ref{P-PWF-0},   $A_F(x)$ is nonempty, compact and convex for each $x\in Q$. Hence $A_F\in\mathcal{V}(Q)$.
  Thus the proof is complete.
\end{proof}

\subsection{Existence and convexity properties of the solution set}
Let $F:M\times M\rightarrow\overline{\IR}$  and  $Q\subseteq M$ satisfy
  the conditions assumed at the beginning of the present section.
We have the following existence result on the solution  of {\rm EP} associated to  the pair $(F,Q)$.

\begin{theorem}\label{Existence-EP-2} Suppose that  $Q$ contains a weak pole $o\in {\rm int}_RQ$ and that  $F$  satisfies   {\rm (H1)}-{\rm (H4)}.
  Then    ${\rm EP}(F,Q)\neq\emptyset$ provided that  $Q$ is compact, or
   assumptions   {\rm (b)} in Proposition \ref{LY-Existence-coercive}  is  satisfied with $A_F$ in place of $A$.
\end{theorem}

\begin{proof}
By hypotheses (H1) and (H2), we see from Proposition \ref{relation-EP-VIP} that
\begin{equation}\label{Existence-EP-20}
{\rm VIP}(A_F,Q)\subseteq{\rm EP}(F,Q).
\end{equation}
Moreover, by hypotheses (H3) and (H4), we get by Proposition \ref{P-PWF} (ii) that $A_F\in\mathcal{V}(Q)$. 
Thus, 
by assumption, Proposition \ref{LY-Existence-coercive} is applicable to getting that
$$
{\rm VIP}(A_F,Q)\neq\emptyset.
$$
The result follows immediately from \eqref{Existence-EP-20} and the proof is complete.
\end{proof}

\begin{remark}\label{LCN}
Assumption    {\rm (b)} in Proposition \ref{LY-Existence-coercive} is satisfied with $A_F$ in place of $A$ if  $Q$ has the BBC and one of the following assumptions holds {(in particular, assumption (b2) was used by Colao et al in \cite{Colao2012})}:
\begin{enumerate}
                                           \item[{\rm (b1)}]  $A_F$ satisfies the coerciveness condition on $Q$.

                                           \item[{\rm (b2)}] There exists a compact set $L\subseteq M$ such that
\begin{equation}\label{LCN-0}
      \begin{array}{l}
        x\in Q\setminus L\Rightarrow
        [\exists y\in Q\cap L\mbox{ s.t.  }F(x,y)<0].
      \end{array}
      \end{equation}

                                         \end{enumerate}
{In fact, it is clear from Remark \ref{remark-co} in the case of (b1). 
To check this for the case of (b2), let}
$L\subseteq M$, $x\in Q\setminus L$ and let  $y\in Q\cap L$ be  given by  \eqref{LCN-0} such that  $F(x,y)<0$.
Then,
\begin{equation}\label{co3.7}
F(x,y)-F(x,x)\le F(x,y)<0.
\end{equation}
  By assumption (H2) and the definition of $A_F$, we see that for any $v\in A_F(x)$, there exists a minimal geodesic $\gamma_{xy}\in{\rm min-}\Gamma^Q_{xy}$ such that $F(x,y)\ge F(x,x)+\langle v, \dot{\gamma}_{xy}(0)\rangle.$
This, together with \eqref{co3.7}, implies that $\langle v, \dot{\gamma}_{xy}(0)\rangle<0$. Hence condition \eqref{L_compact} 
is satisfied as $x\in Q\setminus L$ is arbitrary, and  the proof is  complete.
\end{remark}

The following theorem provides the convexity properties of the solution set ${\rm EP}(F,Q)$, which is a direct consequence of  Propositions \ref{relation-EP-VIP}, \ref{P-PWF} and \ref{LY-unique} (noting by (3.7) that $Q$ is weakly convex).

\begin{theorem}\label{EP-uinique}
Suppose that $F$ satisfies {\rm (H1)}-{\rm (H3)} and ${\rm EP}(F,Q)\not=\emptyset$.
 Suppose further that   $A_F$   is monotone on $Q$  with \eqref{f00} (e.g.,  $F$ is monotone on $Q\times Q$). Then the following assertions hold:

 {\rm (i)} The solution set ${\rm EP}(F,Q)$ is locally convex.

 {\rm (ii)} If  $A_F$  is strictly monotone on $Q$ (e.g., $F$ is strictly monotone on $Q\times Q$),  then ${\rm EP}(F,Q)$ is a singleton.

 {\rm (iii)}  If $M$ is of the sectional curvatures bounded above by   $\kappa>0$, then ${\rm EP}(F,Q)$ is $D_\kappa$-convex.
\end{theorem}



In particular, in the case when $M$ is a Hadamard manifold, hypothesis (H1) implies (H2), and every convex subset has both weak poles and the BCC property.
Thus   the following corollary  is immediate from Theorems \ref{Existence-EP-2} and \ref{EP-uinique}.

\begin{corollary}\label{Existence-EP-uni-3-H}
Let $M$ be a Hadamard manifold. Suppose that $F$ satisfies {\rm (H1)} and {\rm (H3)}.  Then the following assertions hold:

{\rm (i)} If {\rm (H4)}   holds, then ${\rm EP}(F,Q)\neq\emptyset$ provided that  $Q$ is  compact,  or one of 
{\rm (b1)} and {\rm (b2)} in Remark
\ref{LCN} holds.

{\rm (ii)} If $F$ is monotone on $Q\times Q$ with ${\rm EP}(F,Q)\neq\emptyset$, then ${\rm EP}(F,Q)$ is convex.

\end{corollary}

\begin{remark}\label{remark-colao-c}
Assertion {\rm (ii)}  in Corollary \ref{Existence-EP-uni-3-H} seems new even in the Hadamard manifold setting; while assertion {\rm (i)} was established in \cite[Theorem 3.2]{Colao2012} under the following assumptions:

\begin{description}
  \item[{\rm (c1)}]  there exists a compact set $L\subseteq M$ and  $y_0\in Q\cap L$
  such that
 $F(x,y_0)<0$   $\forall x\in Q\setminus L$;

  \item[{\rm (c2)}] the set $\{y\in Q:F(x,y)<0\}$ is convex for each $x\in Q$.
\end{description}

Clearly assumption {\rm (c1)} implies our assumption {\rm (b2)} in Remark  \ref{LCN}. Moreover,  as will be seen  in the application to the  proximal point algorithm in the next section  and to the mixed variational inequalities in Subsection 5.2, assumption (c2) is not satisfied, in general (thus \cite[Theorem 3.2]{Colao2012} is not applicable); while Corollary \ref{Existence-EP-uni-3-H} is applicable because our assumptions {\rm (H1)} and {\rm (H2)} presented here are satisfied there.
\end{remark}

\section{Resolvent and proximal point algorithm for EP}
As in the previous section, we always assume that  $F:M\times M\rightarrow\overline{\IR}$ and $Q\subseteq M$ satisfy the conditions assumed at the beginning of Section 3. 
Recall the equilibrium problem   is defines by \eqref{EP} 
and its solution set is denoted by ${\rm EP}(F,Q)$.
{The aim of this section is to introduce the resolvent and the proximal point algorithm for EP \eqref{EP} on general manifolds and show convergence of this algorithm. The applications of the proximal point method to solve many different problems in the Riemannian context could be fond in  e.g., \cite{Li2009,LiY2012,Bacak2013,Ferreira2002}.

Fix  $z\in M$} and define 
the bifunction $G_z:M\times M\rightarrow\overline{\IR}$ by
$$
G_z(x,y):=\sup_{u\in \exp_x^{-1}z,v\in \exp_x^{-1}y}\langle -u,v\rangle_x\quad\mbox{for any }(x,y)\in M\times M.
$$
In the following definition, we extend the notion of the resolvent defined in \cite[definition 4.6]{Colao2012} for the bifunction $F$  on  Hadamard manifolds to the general manifold setting. Let $\lambda>0$.
\begin{definition}\label{resolvent}
  The resolvent   $J_{\lambda}^F:M\rightrightarrows Q$ of $F$ is defined    by
\begin{equation}\label{J_F-def}
J_{\lambda}^F(z):={\rm EP}(F_{\lambda,z},Q)\quad\mbox{for any } z\in M,
\end{equation}
where the bifunction $F_{\lambda,z}:M\times M\rightarrow\overline \IR$ is defined as
$$
F_{\lambda,z}(x,y):=\lambda F(x,y)+G_z(x,y)\quad\mbox{for any } (x,y)\in M\times M.
$$
\end{definition}

For the remainder, we always assume that $M$ is of the sectional curvature bounded above by $
\kappa\ge0$. Recall that $D_{\kappa}=\frac{\pi}{\sqrt{\kappa}}$ if $\kappa>0$ and $D_{\kappa}=+\infty$
if $\kappa=0$. Then, for any $z\in M$,  ${\bf B}(z,\frac{D_\kappa}{4})$ is strongly convex (see, e.g., \cite[p. 169]{Sakai1996}), and so
$\exp_{(\cdot)}^{-1}z$ is a singleton on ${\bf B}(z,\frac{D_\kappa}{4})$.

Recall that  $A_F$ is the set valued vector field associated to the bifunction $F$ (see \eqref{VIP-10}).
Following \cite{LiY2012}, the resolvent  $J_{\lambda}^{A_F}:M\rightrightarrows Q$ of ${A_F}$ is defined by
\begin{equation}\label{J_A_F_def}
J_{\lambda}^{A_F}(z):={\rm VIP}(A^F_{\lambda,z},Q) \;\;\;\; \mbox{for any } z\in M,
\end{equation}
where $A^F_{\lambda,z}:M\rightrightarrows {TM}$ is defined by
$$
A^F_{\lambda,z}(x):=\lambda A_F(x)-E^Q_z(x)\;\;\;\;\mbox{for any } x\in Q,
$$
with the set-valued vector field $E^Q_z:M\rightrightarrows{TM}$ defined by
$$
E^Q_z(x):=\{u\in \exp_{x}^{-1}z:\exp_xtu\in Q\;\;\forall t\in[0,1]\}\quad\mbox{for any } x\in Q
$$

The  following theorem provides an estimate for the domain of  the resolvent  $J_\lambda^F$. Recall that $L^F_{z}$ denotes
 the center Lipschitz constant of $F(z,\cdot)$ at $z\in M$.  Set
 \begin{equation}\label{D-definition}
  \mathcal{D}_\lambda^F:=\{z\in Q:\lambda L^F_{z}<\frac{D_{\kappa}}{4}\}.
 \end{equation}

\begin{theorem}\label{APA-EQ}
Suppose that $F$ satisfies hypotheses {\rm (H1)}-{\rm (H3)} and  is monotone on $Q\times Q$. 
Then,

{\rm (i)} $\lambda {\rm d}(0,A_F(z))< \frac{D_{\kappa}}{4}$  for each $z\in \mathcal{D}_\lambda^F$.

{\rm (ii)} $\mathcal{D}_\lambda^F\subseteq \mathcal{D}(J_\lambda^F)$.

{\rm (iii)} $J_\lambda^F(z)\cap{\bf B}(z,\frac{D_\kappa}{4})= J_{\lambda}^{A_F}(z)$ is a singleton for each $z\in \mathcal{D}_\lambda^F$.

\end{theorem}

\begin{proof}
Note that $Q$ is weakly convex by implication \eqref{A-Weak-C}. Moreover, by assumptions made for $F$, one sees by  Propositions \ref{P-PWF-0} and \ref{P-PWF} that $A_F\in\mathcal{V}(Q)$   is monotone, and
$${\rm d}(0,A_F(z))\le  L_z^F\quad \mbox{for each }z\in Q.$$
Thus  assertion (i) follows from the definition of $\mathcal{D}_\lambda^F$ in \eqref{D-definition}.
 Below we show assertions (ii) and (iii). To do this,  let $z\in\mathcal{D}_\lambda^F$. Then, $\lambda {\rm d}(0,A_F(z))< \frac{D_{\kappa}}{4}$ by
 (i), and it follows from  \cite[Lemma 4.3 and Corollary 5.4]{LiY2012} (applied to $A_F,\, Q$ in place of $V,\,A$ there) that
\begin{equation}\label{LY-APA-B-10}
\mbox{ $J^{A_F}_\lambda(z)$ is a singleton  and }  J^{A_F}_\lambda(z)\subseteq{\bf B}(z,\frac{D_\kappa}{4}).
\end{equation}
Moreover, thanks to hypotheses (H1) and  (H2), we have by Proposition \ref{example} (i) and (iii)
that  the bifunction $F_{\lambda,z}=\lambda F +G_z $ satisfies 
hypotheses (H1) and (H2) (with $F_{\lambda,z}$ in place of $F$). Furthermore, for any $x\in Q$, $F(x,x)=0$ by
the monotonicity of $F$ and $G_z(x,x)=0$ by of  Proposition \ref{example} (i); hence
$F_{\lambda,z}(x,x)=0$.
Thus one can apply Proposition \ref{relation-EP-VIP} to get  that
\begin{equation}\label{subdiff-EQ-0}
{\rm EP}(F_{\lambda,z},Q)={\rm VIP}(A_{F_{\lambda,z}},Q).
\end{equation}
Noting that $\mathcal{D}(G_{z}(x,\cdot))=M$ by Proposition \ref{example} (i), we see from Lemma \ref{Sub-Sum} and Proposition \ref{example} (iv) that, for any $x\in Q$,
\begin{equation}\label{subdiff-EQ}
A_{F_{\lambda,z}}(x):=\partial (\lambda F(x,\cdot) + G_{z}(x,\cdot))(x)=\lambda A_F(x)-\overline{{\rm co}}E^Q_{z}(x).
\end{equation}
 Hence %
$A^F_{\lambda,z}(x)\subseteq A_{F_{\lambda,z}}(x)$  for any  $ x\in Q$, and then ${\rm VIP}(A_{\lambda,z}^F, Q)\subseteq {\rm VIP}(A_{F_{\lambda,z}},Q)$.
By defintion (see \eqref{J_F-def} and \eqref{J_A_F_def}) and  \eqref{subdiff-EQ-0}, it follows that
\begin{equation}\label{LYp-APA-B-1-0}
J^{A_F}_\lambda(z)={\rm VIP}(A_{\lambda,z}^F, Q) \subseteq {\rm VIP}(A_{F_{\lambda,z}},Q)={\rm EP}(F_{\lambda,z},Q)=J_{\lambda}^F(z).
\end{equation}
In light of \eqref{LY-APA-B-10}, we see that $J_{\lambda}^F(z)\neq\emptyset$, and so assertion (ii) holds as $z\in\mathcal{D}_\lambda^F$ is arbitrary.
To show assertion (iii),  note that $A^F_{\lambda,z}(x)= A_{F_{\lambda,z}}(x)$  if ${\rm d}(x, z)<D_\kappa$ by \eqref{subdiff-EQ}. It follows from \eqref{LYp-APA-B-1-0} that
$$J^{A_F}_\lambda(z)\cap{{\bf B}(z,\frac{D_\kappa}{4})}=J_{\lambda}^F(z)\cap{{\bf B}(z,\frac{D_\kappa}{4})}.$$
This, together with \eqref{LY-APA-B-10}, implies that  $J_\lambda^F(z)\cap{\bf B}(z,\frac{D_\kappa}{4})= J_{\lambda}^{A_F}(z)$ is a singleton, and so
 assertion (iii) holds. The proof is complete.
\end{proof}

The following theorem provides sufficient conditions for $\mathcal{D}(J_{\lambda}^F)=M$. In particular, in the Hadamard manifold setting,
this result was claimed in \cite[Theorem 4.9]{Colao2012} under  the additional assumption  (c1) in Remark \ref{remark-colao-c}
but the proof presented there is not correct.

\begin{theorem}\label{resolvent-FD-well-M}
Suppose that $F$ satisfies hypotheses {\rm (H1)}-{\rm (H3)} and  is monotone on $Q\times Q$. Then,  $\mathcal{D}(J_{\lambda}^F)=M$ provided that one of the following assumptions holds:

{\rm (a)}  $Q$ is compact and contains a weak pole $o\in {\rm int}_RQ$;

{\rm (b)}   $M$ is a Hadamard manifold.
\end{theorem}

\begin{proof}  Let $z\in M$.   Then by the assumptions made for $F$ and Proposition \ref{example} (i), (iii) and (v), one can checks easily
that  the bifunction $F_{\lambda,z}=\lambda F +G_z$ satisfies (H1)-(H4). To complete the proof, it suffices to verify that $J_{\lambda}^F(z)\not=\emptyset$,
which is true  by Theorem \ref{Existence-EP-2} in case {\rm (a)}. Thus we only consider case {\rm (b)}.
To  do this,  we assume that
  $M$ is a Hadamard manifold. Then, for any $x,y\in M$, $\exp_x^{-1}y$ is a singleton and 
${F_{\lambda,z}}(x,y)$ is reduced to
$$F_{\lambda,z}(x,y):=\lambda F(x,y)-\langle\exp_x^{-1}z,\exp_x^{-1}y\rangle.$$
Recalling from \cite[(2.7)]{Colao2012} that
$$
\langle\exp_x^{-1}z,\exp_x^{-1}y\rangle+\langle\exp_y^{-1}w,\exp_y^{-1}x\rangle\ge{\rm d}^2(x,y)\quad\mbox{for any }x,y\in M
$$
  and that $F$ is monotone on $Q\times Q$, we get that
\begin{equation}\label{EXP-Ha-1}
F_{\lambda,z}(x,y)+F_{\lambda,z}(y,x)\le -{\rm d}^2(x,y)\quad\mbox{for any }(x,y)\in Q\times Q.
\end{equation}
Below we show that $F_{\lambda,z}$ satisfies (b2) in Remark \ref{LCN}: there is a compact subset $L\subseteq M$ such that
\begin{equation}\label{F-lambda-b2}
\mbox{$x\in Q\setminus L$ $\Rightarrow$
[$\exists$ $y\in Q\cap L$ $s.t.$ $F_{\lambda,z}(x,y)<0$]}.
\end{equation}
Granting this, we get $J_{\lambda}^F(z)={\rm EP}(F_{\lambda,z},Q)\neq\emptyset$ by Corollary \ref{Existence-EP-uni-3-H},  and  the proof is complete.
 To  show \eqref{F-lambda-b2}, take $y\in Q$ and set $R:=L_{y}^{F_{\lambda,z}}$. Then $R<+\infty$  as $F_{\lambda,z}$ satisfies (H1), and  $L:=\overline{{\bf B}(y,R)}$ is as desired. To show this, let $x\in Q\setminus L$ and  $v\in A_{F_{\lambda,z}}(y)$.  
  Then, ${\rm d}(x,y)>R$, and $\|v\|\le R$ by Proposition \ref{P-PWF-0}. Therefore, we have that
$$
F_{\lambda,z}(y,x)\ge F_{\lambda,z}(y,y)+\langle v,\exp_{y}^{-1}x\rangle\ge -R{\rm d}(x,y)
$$
(noting that $F_{\lambda,z}(y,y)=0$).
This, together with \eqref{EXP-Ha-1}, implies
$$F_{\lambda,z}(x,y)\le -{\rm d}^2(x,y)-F_{\lambda,z}(y,x)\le \left(R-{\rm d}(x,y)\right){\rm d}(x,y)<0.$$
Thus, \eqref{F-lambda-b2} is shown, and the proof is complete. 
\end{proof}

To define the proximal point algorithm for solving EP \eqref{EP},   let $x_0\in Q$ and $\{\lambda_k\}\subset(0,+\infty)$. Thus the proximal point algorithm
  can be formulated as  follows.

{\bf Algorithm P}   Letting $k=1,2,\dots$ and having $x_k$, choose $x_{k+1}$
such that
$$
x_{k+1}\in J_{\lambda_k}^F(x_k)\cap{{\bf B}(x_k,\frac{D_{\kappa}}{4})}.
$$
Clearly, in the case when $M$ is a Hadamard manifold, {\bf Algorithm P} is reduced to the one defined in  \cite{Colao2012}:
$$
x_{k+1}\in J_{\lambda_k}^F(x_k)\quad\mbox{for each }k\in \IN.
$$

The  convergence result of {\bf Algorithm P} is as follows.



\begin{theorem}\label{Convergence-APA-EP}
Suppose that  $F$ satisfies hypotheses {\rm (H1)}-{\rm (H3)} and  is monotone on $Q\times Q$ with   ${\rm EP}(F,Q)\neq\emptyset$. Let $x_0\in Q$ and  $\{\lambda_k\}\subset (0,\infty)$ be such that
\begin{equation}\label{C-APA-EP00}
{\rm d}(x_0,{\rm EP}(F,Q))< \frac{D_{\kappa}}{8},
\end{equation}
\begin{equation}\label{C-APA-EP0}
\Sigma_{k=0}^{\infty}\lambda_k^2=\infty\quad\mbox{and}\quad \lambda_k L^F_{x_k}<\frac{D_{\kappa}}{4} \quad\mbox{for all } k\in\IN.
\end{equation}
Then, {\bf Algorithm P} is well-defined, and   converges to a point in ${\rm EP}(F,Q)$.
\end{theorem}

\begin{proof}
Recall that  $A_F:Q\rightrightarrows {TM}$ is defined by \eqref{VIP-10}. By assumption, Propositions \ref{relation-EP-VIP} and \ref{P-PWF} are applicable; hence   $A_F$ is monotone, $A_F\in\mathcal{V}(Q)$, and
\begin{equation}\label{Convergence-1}
{\rm VIP}(A_F,Q)={\rm EP}(F,Q)
\end{equation}
(noting  that \eqref{f00} hold by the  monotonicity assumption).
Then, thanks to  \eqref{C-APA-EP00}, one sees that
\begin{equation}\label{pConvergence-1}
{\rm d}(x_0,{\rm VIP}(A_F,Q))<\frac{D_{\kappa}}{8}.
\end{equation}
Let $\{\tilde {x}_k\}$ be a sequence generated by the following proximal algorithm with initial point $\tilde x_0:=x_0$, which was introduced in \cite{LiY2012} for finding a point in ${\rm VIP}(A_F,Q)$:
\begin{equation}\label{APA-VIP}
\tilde x_{k+1}\in J_{\lambda_k}^A(\tilde x_k) \quad \mbox{for each }k\in \IN.
\end{equation}
In view of the second assumption in  \eqref{C-APA-EP0},  and applying Theorem \ref{APA-EQ} (iii), we can check  inductively    that {\bf Algorithm P} is well-defined and that the generated sequence $\{{x}_k\}$ coincides with $\{\tilde {x}_k\}$  and satisfies
$$
\lambda_k {\rm d}(0,A_F(\tilde x_k))< \frac{D_{\kappa}}{4}\quad\mbox{for each }k\in\IN.
$$
This, together with the first assumption in \eqref{C-APA-EP0} and \eqref{pConvergence-1}, implies that \cite[Corollary 5.8]{LiY2012} (with $A_F$, $Q$ in place of $V$, $A$) is applicable, and the sequence $\{\tilde x_k\}$  and so $\{x_k\}$  converges to a point in ${\rm VIP}(A_F,Q)$. Thus the conclusion follows    immediately from \eqref{Convergence-1}, and the proof is complete.
\end{proof}

In the special case when  $M$ is a Hadamard manifold, assumption \eqref{C-APA-EP00} and the second one in \eqref{C-APA-EP0} are satisfied automatically. Therefore the following corollary is
direct from Theorem \ref{Convergence-APA-EP}, which was claimed in    \cite[Theorem 4.9, 4.10]{Colao2012} (for constant parameters $\lambda_k\equiv \lambda>0$) but with an incorrect proof there   as we explained  in Section 1).

\begin{corollary}\label{APA-coro}
Suppose that  $M$ is a Hadamard manifold, and that  $F$ satisfies hypotheses {\rm (H1)}-{\rm (H3)} and  is monotone on $Q\times Q$ with   ${\rm EP}(F,Q)\neq\emptyset$. Let  $\{\lambda_k\}\subset (0,\infty)$ be such that
$\Sigma_{k=0}^{\infty}\lambda_k^2=\infty$.
Then, {\bf Algorithm P} is well-defined, and converges to a
solution in ${\rm EP}(F,Q)$.
\end{corollary}

\section{Applications} This section is devoted to two applications of the results regarding the solution set of the EP in the previous sections:
One is to the Nash equilibrium and the other to the mixed variational inequality.

\subsection{Nash equilibrium}

We consider the Nash equilibrium problem (NEP for short) on Riemannian manifolds in this subsection, which is formulated as follow. Let $I=\{1,2,\dots,m\}$ be a finite index set which denotes the set of players, and let $(M_i,{\rm d}_i), \,  i\in I$, be a Riemannian manifold. For each $i\in I$, let  $Q_i\subseteq M_i$ be the strategy set of the $i$-th player, and $f_{i}:M\rightarrow\overline{\IR}$ be his loss-function, where $M:=M_1\times M_2\times\dots \times M_m$ is the product manifold with the standard Riemannian product metric. The Nash equilibrium problem associated
to $Q:=Q_1\times Q_2\times\dots \times Q_m\subseteq M$ and $\{f_i\}_{i\in I}$ consists of finding a point ${\bar x} =({\bar x} _i)\in Q$ such that
\begin{equation}\label{Nash-E}
f_i({\bar x} )=\min_{y_i\in Q_i}f_{i}({\bar x} _1,\dots,{\bar x} _{i-1},y_i,{\bar x} _{i+1},\dots,{\bar x} _m) \quad\mbox{for each $i\in I$}.
\end{equation}
Any point ${\bar x} \in Q$ satisfying \eqref{Nash-E} is called a Nash equilibrium point of the NEP, and we denote the set of all Nash equilibrium points by ${\rm NEP}(\{f_i\}_{i\in I},Q)$.



The most well-known existence results for the classical NEP in the linear space setting is due to Nash \cite{Nash1950,Nash1951},
where it is assumed that each $Q_i$ is  compact and convex  and each  $f_i$ is
(quasi)convex in the $i$-th variable.
Further extensions and applications of Nash's original work could be founded in
\cite{Georgiev2005,Morgan2007,Nessah2008,Tala1996} and references therein.
Krist\'{a}ly seems  the first one   to consider the existence and localization of NEP
in the framework of Riemannian manifolds; see  [23]. Recently, Krist\'{a}ly      used  in  [24] a variational approach to analyzed the NEP
with nonconvex strategy sets and nonconvex /nonsmooth payoff functions in Hadamard manifolds.

To proceed, we   assume for  the whole subsection that

%
%
\begin{description}
 \item[(${\rm H_N}$-a)] $Q_i$ is closed and weakly convex in $M_i$ and $Q\subseteq{\rm int}\bigcap_{i\in I}\mathcal{D}(f_{i})$;
 \item[(${\rm H_N}$-b)]  for each $i\in I$, $f_i$ is continuous on $Q$;
 \item[(${\rm H_N}$-c)] for each $i\in I$, $f_{i}$ and $f_{i}+\delta_{Q}$ are  weakly convex in the $i$-th variable.
  \end{description}

To apply our results in the previous sections, we,  following \cite{Colao2012} and \cite{Rosen1965},
 reformulate NEP \eqref{Nash-E} as an EP as follows. 
Let $r:=(r_i)\in \IR_{++}^m:=\{(r_i)\in \IR^m: \mbox{each }r_i>0\}$, and define the bifunction $F_r:M\times M\rightarrow \mathbb{R}$ as
 the weighted positive sum of the functions $\{f_i\}$:
$$
F_r(x,y):=\sum_{i\in I}r_i\left(f_i(x_1,\ldots,x_{i-1},y_i,x_{i+1},\ldots,x_n)-f_i(x)\right)
$$
 for any $x=(x_i)_{i\in I}, y=(y_i)_{i\in I}\in M$. Then, it is easy to check  that
\begin{equation}\label{Nash-E-E-r}
\mbox{${\rm EP}(F_r,Q)$=${\rm NEP}(\{f_i\}_{i\in I},Q)$}.
\end{equation}
  In the spirit of the idea in \cite{Rosen1965} for the NEP in the Euclidean space setting,
   we introduce the pseudosubgradient mapping $g_r:M\to TM$ for functions $\{f_i\}$ in the Riemannian manifold setting, which is defined by
$$
g_r(x):=\left(r_1\partial_1f_1(x),r_2\partial_2f_2(x),\dots,r_m\partial_mf_m(x)\right)\quad\mbox{for each }x\in M,
$$
where,  for each $i\in I$ and $x\in M$, $\partial_i f_{i}(x)$
stands for the subdifferential of the function  $f_i(x_1,\dots,x_{i-1},\cdot,x_{i+1},\dots,x_m)$ at $x_i$, that is
$$\partial_i f_{i}{(x)}:=\partial f_i(x_1,\dots,x_{i-1},\cdot,x_{i+1},\dots,x_m)(x_i).$$
%
By definition, we check that
\begin{equation}\label{AF-NEPE-S}
A_{F_r}(x):=\partial {F_r}(x,\cdot)(x)=g_r(x)\quad\mbox{ for any }x\in Q.
\end{equation}
The main theorem in this subsection is as follows, which provides the results on
the existence, the uniqueness and the convexity of the Nash equilibrium point.

\begin{theorem}\label{Existence-Nash} Let $r\in \IR_{++}^m$. Then the  following assertions hold:

{\rm (i)} Suppose that $Q$ 
contains a weak pole $o\in{\rm int}_RQ$.   Then ${\rm NEP}(\{f_i\}_{i\in I},Q)\neq\emptyset$ provided that $Q$ is compact,
 or $Q$ has the BCC property and that there exists    a compact subset $L\subseteq M$ such that
      \begin{equation}\label{L_compact-o}
        x\in Q\setminus L\Rightarrow[\forall v\in g_r(x),\exists y\in Q\cap L,\gamma_{xy}\in{\rm min-}\Gamma_{xy}^Q \mbox{ s.t.
         }\langle v,\dot{\gamma}_{xy}(0)\rangle<0].
        %
      \end{equation}

{\rm (ii)} If $g_r$ is strictly monotone on $Q$, then ${\rm NEP}(\{f_i\}_{i\in I},Q)$ is at most a singleton.

{\rm (iii)} If $g_r$ is monotone on $Q$  and ${\rm NEP}(\{f_i\}_{i\in I},Q)\neq\emptyset$, then
${\rm NEP}(\{f_i\}_{i\in I},Q)$ is locally convex, and ${\rm NEP}(\{f_i\}_{i\in I},Q)$ is $D_\kappa$-convex if $M$ is additionally assumed
to be of the sectional curvatures bounded above by some  $\kappa\ge 0$.
\end{theorem}

\begin{proof}
In view of \eqref{Nash-E-E-r}, \eqref{AF-NEPE-S}, \eqref{L_compact-o} and  thanks to Theorems \ref{Existence-EP-2} and \ref{EP-uinique} (applied to $F_r$ in place of $F$), it suffices to show
  that $F_r$ satisfies hypotheses (H1)-(H4) made in Section 3. Note that (H3) follows trivially from assumption (${\rm H_N}$-b);
  while (H4)   is clear as  ${F_r(x,x)}=0$ for any $x\in Q$. Thus we only need to show that $F_r$ satisfies hypotheses (H1) and (H2).
To do this, let $x=(x_i)_{i\in I}\in Q$, and write
$$\mathcal{D}_i:=\mathcal{D}(f_i(x_1,\dots,x_{i-1},\cdot,x_{i+1},\dots,x_m))\quad \mbox{ for each }i\in I.
$$
Then
\begin{equation}\label{EQ-DORM}
{\mathcal{D}(F_r(x,\cdot))}=\mathcal{D}_1\times\ldots\times
\mathcal{D}_i\times\ldots\times\mathcal{D}_m.
\end{equation}
By assumption (${\rm H_N}$-a), each $Q_i\subseteq {\rm int} \mathcal{D}_i$ and so
$$
x\in Q\subseteq {\rm int}{\mathcal{D}(F_r(x,\cdot))}.
$$
Furthermore,  in light of   assumption (${\rm H_N}$-c), one sees that  each $\mathcal{D}_i$ is weakly convex in $M_i$.
 This, together with \eqref{EQ-DORM}, implies that   ${\mathcal{D}(F_r(x,\cdot))}$ is weakly convex in $M$. We claim that
 $F_r(x,\cdot)$ and $F_r(x,\cdot)+\delta_{Q\times Q}(x,\cdot)$ are weakly  convex in $M$. Granting this, (H1) and (H2) are checked.
 %
%
In fact, let
  $y=(y_i),  z=(z_i)\in {\mathcal{D}(F_r(x,\cdot))}$. Then,  by assumption (${\rm H_N}$-c), for each $i\in I$,
there is a geodesic $\gamma_i\in {\rm min-}\Gamma_{z_iy_i}^{\mathcal{D}_i}$
such that
 \begin{equation}\label{f-i-dc}
  f_i(x_1,\dots,x_{i-1},\cdot,x_{i+1},\dots,x_m)\circ\gamma_i\mbox{ is convex on }[0,1].
 \end{equation}
Define $\gamma_{zy} [0,1]\rightarrow M$ by   $\gamma_{zy}(t):=(\gamma_1(t),\gamma_2(t),\ldots,\gamma_m(t))$ for each $t\in [0,1]$. Then,
$\gamma_{zy}\in{\rm min-}\Gamma_{zy}^{ F_r(x,\cdot)}$ (see, e.g., \cite{Bridson1999}), and
$$F_r(x,\cdot)\circ\gamma_{zy}=\sum_{i\in I}f_i(x_1,\dots,x_{i-1},\cdot,x_{i+1},\dots,x_m)\circ\gamma_i.$$
This means that $F_r(x,\cdot)\circ\gamma_{zy}$ is clearly convex thanks to \eqref{f-i-dc}, and
so $F_r(x,\cdot)$ is weakly  convex in $M$. Similarly, one can checks that $F_r(x,\cdot)+\delta_{Q\times Q}(x,\cdot)$ is also weakly  convex in $M$.
 Thus the claim stands, and the proof is complete.
%
%
\end{proof}

\begin{remark}
Assertion {\rm (i)} extends the corresponding one in \cite[Theorem 1.1]{Kristaly2010},
 which was proved under the assumption that each $Q_i$ is compact and geodesic convex. It is worthy remarking that  the geodesic convexity assumption for  $Q_i$ prevents
 its application to  some special but important Riemannian manifolds, such as compact Stiefel manifolds ${\rm St}(p,n)$ and Grassmann manifolds ${\rm Grass}(p,n)$ ($p<n$), in which there is no geodesic convex subset (see \cite[p. 104 (5.27)]{Absil2008}). 
\end{remark}

{
 Example \ref{EX5.1} below   provides the case where our existence result of Theorem \ref{Existence-Nash} is applicable but not \cite[Theorem 1.1]{Kristaly2010}. Note also that the NEP in Example \ref{EX5.1} is originally   defined on the Euclidean space, and   the corresponding existing results in the Euclidean space setting  (see, e.g., \cite{Georgiev2005,Morgan2007,Nash1950,Nash1951}), to the best our knowledge, are nor applicable because the set $Q_2$ involved is not convex in the usual sense.  
%
}
{
\begin{example}\label{EX5.1} Consider the  Nash equilibrium problem \eqref {Nash-E} with the associated
  $Q:=Q_1\times Q_2 \subseteq \IR\times \IR^3$ and $\{f_i\}_{i=1,2}$ defined respectively  as follows:
 $$Q_1:=[-1,1],\quad
Q_2:=\{(t_1,t_2,t_3):t_1^2+t_2^2+t_3^2=1,t_1>0,|t_2|\le\frac12,t_3>0\},$$ 
$$f_1(x_1,x_2)=(x_1-t_3)^2\quad\mbox{and}\quad f_2(x_1,x_2)=\arccos t_1\quad\mbox{for any }x_1\in \IR,\;x_2=(t_1,t_2,t_3)\in \IR^3.$$
Clearly $Q_2\subset\IR^3$ is not convex, and so  the existence results in the Euclidean space setting are not applicable. 

Below, we shall consider the problem on the Riemannian manifold $M:=\IR\times \mathbb{S}^2$, where 
\begin{equation*}
\mathbb{S}^2:=\big\{(t_1,t_2,t_3)\in\IR^3\big|\;t_1^2+t_2^2+t_3^2=1
\big\}
\end{equation*}
is the $2$-dimensional unit sphere. Denote ${\bf x}:=(0,0,1)$, ${\bf y}:=(0,0,-1)$, and consider
system
of coordinates $\Phi\colon(0,\pi)\times[0,2\pi]\subset\IR^2\to
\mathbb{S}^2\setminus\{{\bf x},{\bf y}\}$ around 
$x\in\mathbb{S}^2\setminus\{{\bf x},{\bf y}\}$ 
%
defined by
$$\Phi(\theta,\varphi):=(\sin\theta\cos\varphi,\sin\theta\sin\varphi,\cos\theta)^T
\quad\mbox{for each } (\theta \varphi)\in(0,\pi)\times[0,2\pi].
%
$$
  Then the Riemannian metric on $\mathbb{S}^2\setminus\{{\bf x},{\bf y}\}$ is
given by
\begin{equation*}
g_{11}=1,\quad g_{12}=0,\quad g_{22}=\sin^2\theta\quad\mbox{for
each }\;\theta\in(0,\pi)\;\mbox{ and }\;\varphi\in[0,2\pi],
\end{equation*}
and the geodesics of $\mathbb{S}^2\setminus\{{\bf x},{\bf y}\}$ are
great circles or semicircles; see  \cite[p. 84]{Udriste1994} for more details. 

Restricting $f_1$ and  $f_2$ to $M=\IR\times\mathbb{S}^2$, 
  one can check by definition  that assumptions (${\rm H_N}$-a)-(${\rm H_N}$-c) are satisfied (noting that $\arccos t_1={\rm d}(x_2,z_0)$ for each $x_2=(t_1,t_2,t_3)\in \mathbb{S}^2$, where $z_0:=(1,0,0)$), and that $Q\subset M$ is compact and has a weak pole in $ {\rm int}Q$. Thus,  Theorem \ref{Existence-Nash} is applicable and  guarantees
   ${\rm NEP}(\{f_1,f_2\},Q_1\times Q_2)\not=\emptyset$. Indeed, by a simple calculation, we see that ${\rm NEP}(\{f_1,f_2\},Q_1\times Q_2)=\{(1,(1,0,0))\}$. However, the existence result in \cite[Theorem 1.1]{Kristaly2010} is not applicable because there is no geodesic subset on $\mathbb{S}^2$.

\end{example}
}

As explained before  (see the paragraph right before Corollary \ref{Existence-EP-uni-3-H}), the following corollary is a direct consequence of Theorem \ref{Existence-Nash}. In particular, assertion (i) was proved in \cite[Theorem 3.12]{Colao2012} with each $Q_i$ being compact; while assertion (ii) is new even in the Hadamard manifold setting.

\begin{corollary}\label{Existence-Nash-H}
Suppose that each $M_i$ is a Hadamard manifold. Then, the following assertions hold:

{\rm (i)} The solution set ${\rm NEP}(\{f_i\}_{i\in I},Q)\neq\emptyset$ provided $Q$ is compact, or there exists
 a compact subset $L\subseteq M$ such that \eqref{L_compact-o} holds for some $r\in \IR_{++}^m$.

{\rm (ii)} If there exists some $r\in \IR_{++}^m$ such that $g_r$ is monotone on $Q$, then ${\rm NEP}(\{f_i\}_{i\in I},Q)$ is convex.
\end{corollary}


\begin{remark}\label{Nash-CP1}  In view of \eqref{AF-NEPE-S}, one checks by Remark \ref{LCN} (applied to $g_r$, $F_r$ in place of $A_F$, $F$) that
 a compact subset $L$ exists such that
  \eqref{L_compact-o} holds provided one of the following assumptions holds:

{\rm (a)} 
$g_r$ satisfies the coerciveness condition on $Q$;

{\rm (b)} 
there exists a compact set $L\subseteq M$ such that
$$
      \begin{array}{l}
        x\in Q\setminus L\Rightarrow
        [\exists y\in Q\cap L\mbox{ s.t.  }F_r(x,y)<0].
      \end{array}
      $$

\end{remark}


\subsection{Mixed variational inequalities}

Let $Q\subset M$ be a nonempty closed subset. Given a vector field $V:Q\rightarrow {TM}$ and a real-valued function $f:M\rightarrow\overline{\IR}$. 
The mixed variational inequality
problem (MVIP for short) associated to $V$ and $f$  is to  find $\bar x\in Q$, called a solution of the MVIP, such that 
\begin{equation}\label{MVIP}
\langle V(\bar x),\dot{\gamma}_{\bar xy}(0)\rangle+f(y)-f(\bar x)\ge0\quad \mbox{for any $y\in Q$, $\gamma_{\bar xy}\in\Gamma_{\bar xy}^Q$}.
\end{equation}
The set of all solutions of MVIP \eqref{MVIP} is denoted by ${\rm MVIP}(V,f,Q)$. The  MVIP  has been studied extensively in the linear space setting; see, e.g., \cite{Glowinski1981,He1999,WangYangHe2001}; and it seems that \cite{Colao2012} is the first paper to explore
the MVIP in the Hadamard manifold  setting, where only  the  existence  issue  of  the solution for the  MVIP  is concerned with.

To reformulate the MVIP  as an EP considered in the previous sections,  we 
define $F:M\times M\rightarrow(-\infty,+\infty]$ as follows:
\begin{equation}\label{MVIP-F}
F(x,y):=\sup_{u\in \exp_x^{-1}y}\langle V(x),u\rangle_{x}+f(y)-f(x)\quad\mbox{for any } (x,y)\in M\times M,
\end{equation}
where $\exp_x^{-1}y$ is defined by \eqref{EYX} and we adopt the convention that $a-(+\infty)=+\infty$ for any $a\in\IR$. 
\begin{proposition}
Let $F:M\times M\rightarrow(-\infty,+\infty]$ be defined by \eqref{MVIP-F}. Suppose that $f$ is convex. Then we have
\begin{equation}\label{EQ-MVIP}
{\rm MVIP}(V,f,Q)={\rm EP}(F,Q).
\end{equation}
\end{proposition}

\begin{proof}
It is evident that ${\rm MVIP}(V,f,Q)\subseteq{\rm EP}(F,Q)$. To show the converse inclusion,
let $\bar x\in{\rm EP}(F,Q)$ and   it suffices to prove that \eqref{MVIP} holds.
To this end, let $y\in Q$ and  $\gamma_{\bar xy}\in\Gamma_{\bar xy}^Q$. We have to show that
\begin{equation}\label{EQ-MVIP-r2}
\langle V(\bar x),\dot{\gamma}_{\bar xy}(0)\rangle+f(y)-f(\bar x)\ge0.
\end{equation}
Take $\bar t\in(0,1]$   such that ${\rm d}(\bar x,\gamma_{\bar xy}(\bar t))\le r_{\bar x}$ (note that $r_{\bar x}>0$ by \eqref{cvx-rd-p}). Denote $\bar y :=\gamma_{\bar xy}(\bar t)$. Then  $\bar y\in Q$ and $\Gamma_{\bar x\bar y}^Q=\{\gamma_{\bar x\bar y}\}$ is a singleton, where 
   $\gamma_{\bar x\bar y}:[0,1]\rightarrow M$ is defined by $$\gamma_{\bar x\bar y}(s):=\gamma_{\bar xy}(\bar ts)\quad\mbox{for any }s\in[0,1].$$
   Then $\dot{\gamma}_{\bar x\bar y}(0)=\bar t\dot{\gamma}_{\bar xy}(0)$. 
In view of $\bar x\in{\rm EP}(F,Q)$ and $\bar y\in Q$, we see that
 $$
 \langle V(\bar x),\dot{\gamma}_{\bar x\bar y }(0)\rangle+f(\bar y)-f(\bar x)= \langle V(\bar x),\bar t\dot{\gamma}_{\bar xy }(0)\rangle+f(\bar y)-f(\bar x)\ge0.
$$
Noting that  $\frac{f(\bar y)-f(\bar x)}{\bar t}\le f(y)-f(\bar x)$ by the convexity of 
$f\circ {\gamma}_{\bar xy }$ (as $f$ is convex), we conclude  that \eqref{EQ-MVIP-r2} holds, which completes the proof. 
\end{proof}

We assume in the present subsection that

($\rm H_M$-a) $f$ is convex and $Q\subseteq{\rm int}\mathcal{D}(f)$ is closed   weakly convex. 

($\rm H_M$-b) $V $ is 
continuous on $Q$.


The following theorem gives the existence, the uniqueness and the convexity property about the solution set ${\rm MVIP}(V,f,Q)$.

\begin{theorem}\label{Existence-MVIP}
The following assertions hold:

{\rm (i)} Suppose that $Q$ 
contains a weak pole $o\in{\rm int}_RQ$. Then ${\rm MVIP}(V,f,Q)\neq\emptyset$ provided that  $Q$ is compact, or  $Q$ has the BCC property and there exists
 a compact subset $L\subseteq M$ such that
      \begin{equation}\label{L_compact-o-M}
        x\in Q\setminus L\Rightarrow[\forall v\in \partial f(x),\exists y\in Q\cap L,\gamma_{xy}\in{\rm min-}\Gamma_{xy}^Q \mbox{ s.t.
         }\langle V(x)+v,\dot{\gamma}_{xy}(0)\rangle<0].
      \end{equation}
%
%

{\rm (ii)} If $V+\partial f$ is strictly monotone on $Q$, then ${\rm MVIP}(V,f,Q)$ is at most a  singleton.

{\rm (iii)} If   $V+\partial f$ is monotone on $Q$ and ${\rm MVIP}(V,f,Q)\neq\emptyset$,
then ${\rm MVIP}(V,f,Q)$ is locally convex, and  is $D_\kappa$-convex
if $M$ is additionally assumed to be of the sectional curvatures bounded above by some $\kappa\ge 0$.
\end{theorem}

\begin{proof}
We first  show that $F$ satisfies hypotheses (H1)-(H4) made in Section 3. To do this, let $G_V:M\times M\rightarrow \IR$ be defined by \eqref{f-G-G}, and let $G:M\times M\rightarrow \IR$ be defined by
\begin{equation}\label{F-1-G}
G(x,y):=f(y)-f(x)\quad\mbox{for any } (x,y)\in Q\times M.
\end{equation}
Then $F=G_V+G$. Noting  by assumption ($\rm H_M$-a) that  both $G$ and $G+\delta_{Q\times Q}$ are point-wise weakly convex on $Q$, we see
from Proposition \ref{example}(iii) that $F=G_V+G$ and $F+\delta_{Q\times Q}=G_V+ (G +\delta_{Q\times Q})$ are point-wise weakly convex on $Q$.
This particularly means  that $F$ satisfies (H2). To show (H1) and (H4), recalling \eqref{domain-G-G} in  Proposition \ref{example}(i), one checks
that   
$$\mathcal{D}(F(x,\cdot))=\mathcal{D}(G(x,\cdot))\bigcap\mathcal{D}(G_V(x,\cdot))=\mathcal{D}(f)\quad\mbox{ for any }x\in Q.$$
In view of assumption ($\rm H_M$-a), (H1) is checked; while  (H4) is trivial since, by \eqref{domain-G-G}, ${F(x,x)}=0$ for any $x\in Q$. Thus it remains to   check (H3). Since by assumption ($\rm H_M$-a),   the function $x\mapsto G(x,y)$ is continuous on $Q$ (see Lemma \ref{Convex-f-P} (i)). In view of assumption ($\rm H_M$-b), Proposition \ref{example}(v) is applicable to getting that $x\mapsto G_V(x,y)$ is usc on $Q$ and so is $F$. Thus, (H3) is checked.
Next, we check that
\begin{equation}\label{mix-A-F}
A_F(x)=V(x)+\partial f(x)\quad \mbox{for each } x\in Q,
\end{equation}
where $A_F$ is  defined by \eqref{VIP-10}.
Granting this, one verifies that conditions of Theorems \ref{Existence-EP-2} and  \ref{EP-uinique} are satisfied, and then assertions (i)-(iii) follow by \eqref{EQ-MVIP} (which is valid by assumption {\rm ($\rm H_M$-a)}). 
To show \eqref{mix-A-F},
let $x\in Q$. Then $\partial G_V(x,\cdot)(x)=V(x)$ by Proposition \ref{example}(iv). Thus, by assumption ($\rm H_M$-a), one applies Lemma \ref{Sub-Sum} to
obtain
\eqref{mix-A-F} 
and the proof is complete.
 \end{proof}

With a similar argument that we did for Corollary \ref{Existence-Nash-H}, but using Theorem \ref{Existence-MVIP} in
place of Theorem \ref{Existence-Nash}, we have the following corollary. In particular, assertion (i) was claimed in \cite[Theorem 3.5]{Colao2012} with its proof being incorrect; while assertion (ii) is new even in the Hadamard manifold setting.

\begin{corollary}\label{Existence-MVIP-H}
Suppose that $M$ is a Hadamard manifold. Then, the following assertions hold:

{\rm (i)} The solution set ${\rm MVIP}(V,f,Q)\neq\emptyset $ provided $Q$ is compact, or there exists
 a compact subset $L\subseteq M$ such that \eqref{L_compact-o-M} holds.

{\rm (ii)} If $V+\partial f$ is monotone on $Q$ with ${\rm MVIP}(V,f,Q)\not=\emptyset $, then ${\rm MVIP}(V,f,Q)$ is   convex.
\end{corollary}

\begin{remark}\label{Nash-CP}
Under one of the following assumptions,  a compact subset $L\subseteq M$ exists such that   \eqref{L_compact-o-M}  holds:


{\rm (a)} $V$ satisfies the coerciveness condition on $Q$;

{\rm (b)} $\partial f$ satisfies the coerciveness condition on $Q$ and $V$ is monotone on $Q$.

\noindent {Indeed, in view of assumption ($\rm H_M$-a), we see that $\partial f$ is  monotone on $Q$ by definition of the subdifferential of $f$.
Assuming (b) or (c), it is easy to verify by definition that $V+\partial f$ satisfies the coerciveness condition on $Q$. Thus, one checks that \eqref{L_compact-o-M} is satisfied as we have explained in Remark \ref{remark-co} with $V+\partial f$ in place of $A$.}
\end{remark}




\begin{thebibliography}{10}

\bibitem{Absil2008}
{\sc P.~A. Absil, R.~Mahony, and R.~Sepulchre}, {\em {Optimization Algorithms
  on Matrix Manifolds}}, Princeton University Press, Princeton, NJ, 2008.

\bibitem{Adler2002}
{\sc R.~Adler, J.~P. Dedieu, J.~Margulies, M.~Martens, and M.~Shub}, {\em
  {Newton¡¯s method on Riemannian manifolds and a geometric model for the human
  spine}}, IMA J. Numer. Anal., 22 (2002), pp.~359--390.

\bibitem{Homidan2008}
{\sc S.~Al-Homidan, Q.~H. Ansari, and J.C. Yao}, {\em {Some generalizations of
  Ekeland-type variational principle with applications to equilibrium problems
  and fixed point theory}}, Nonlinear Anal., 69 (2008), pp.~126--139.

\bibitem{Bianchi2005}
{\sc M.~Bianchi, G.~Kassay, and R.~Pini}, {\em {Existence of equilibria via
  Ekeland's principle}}, J. Math. Anal. Appl., 305 (2005), pp.~502--512.

\bibitem{Bianchi1996}
{\sc M.~Bianchi and S.~Schaible}, {\em {Generalized monotone bifunctions and
  equilibrium problems}}, J. Optim. Theory Appl., 90 (1996), pp.~31--43.

\bibitem{Blum1994}
{\sc E.~Blum and W.~Oettli}, {\em {From optimization and variational
  inequalities to equilibrium problems}}, Math. Student, 63 (1994),
  pp.~123--145.

\bibitem{Bridson1999}
{\sc M.~Bridson and A.~Haefliger}, {\em {Metric Spaces of Non-Positive
  Curvature}}, Springer-Verlag, Berlin, 1999.

\bibitem{Burke2001}
{\sc J.~V. Burke, A.~Lewis, and M.~Overton}, {\em {Optimal stability and
  eigenvalue multiplicity}}, Found. Comput. Math., 1 (2001), pp.~205--225.

\bibitem{Castellani2010}
{\sc M.~Castellani, M.~Pappalardo, and M.~Passacantando}, {\em {Existence
  results for nonconvex equilibrium problems}}, Optimization Methods and
  Software, 25 (2010), pp.~49--58.

\bibitem{Chadli2000}
{\sc O.~Chadli, Z.~Chbani, and H.~Riahi}, {\em {Equilibrium problems with
  generalized monotone bifunctions and applications to variational
  inequalities}}, J. Optim. Theory Appl., 105 (2000), pp.~299--323.

\bibitem{Colao2012}
{\sc V.~Colao, G.~L\'{o}pez, G.~Marino, and V.~Mart\'{\i}n-M\'{a}rquez}, {\em
  {Equilibrium problems in Hadamard manifolds}}, J. Math. Anal. Appl., 388
  (2012), pp.~61--77.

\bibitem{Combettes2005}
{\sc P.~L. Combettes and S.A. Hirstoaga}, {\em {Equilibrium programming in
  Hilbert spaces}}, J. Nonlinear Convex Anal., 1 (2005), pp.~117--136.

\bibitem{Carmo1992}
{\sc M.~P. DoCarmo}, {\em {Riemannian Geometry}}, Birkh\"{a}user Boston, Boston
  MA, 1992.

\bibitem{Ferreira2005}
{\sc O.~P. Ferreira, L.~R. lucambio P\'{e}rez, and S.~Z. N\'{e}meth}, {\em
  {Singularities of monotone vector fields and an extragradient-type
  algorithm}}, J. Global Optim., 31 (2005), pp.~133--151.

\bibitem{Flam1997}
{\sc S.~D. Fl{\aa}m and A.S. Antipin}, {\em {Equilibrium programming using
  proximal like algorithms}}, Math. Program., 78 (1997), pp.~29--41.

\bibitem{Georgiev2005}
{\sc P.~Gr. Georgiev}, {\em {Parametric Borwein-Preiss variational principle
  and applications}}, Proc. Amer. Math. Soc., 133 (2005), pp.~3211--3225.

\bibitem{Glowinski1981}
{\sc R.~Glowinski, J.~L. Lions, and R.~Tremolieres}, {\em {Numerical Analysis
  of Variational Inequalities}}, North-Holland, Amsterdam, 1981.

\bibitem{Greene-Wu}
{\sc R.~E. Greene and H.~Wu}, {\em {On the subharmonicity and
  plurisubharmonicity of geodesically convex functions}}, Indiana Univ. Math.
  J., 22 (1973), pp.~641--653.

\bibitem{He1999}
{\sc B.~S. He}, {\em {Inexact implicit methods for monotone general variational
  inequalities}}, Math. Program., 86 (1999), pp.~199--217.

\bibitem{Iusem2009}
{\sc A.N. Iusem, G.~Kassay, and W.~Sosa}, {\em {On certain conditions for the
  existence of solutions of equilibrium problems}}, Math. Program., 116 (2009),
  pp.~259--273.

\bibitem{Iusem2003b}
{\sc A.N. Iusem and W.~Sosa}, {\em {Iterative algorithms for equilibrium
  problems}}, Optimization, 52 (2003), pp.~301--316.

\bibitem{Konnov2009}
{\sc I.~V. Konnov}, {\em {Application of the proximal point method to
  nonmonotone equilibrium problems}}, J. Optim. Theory Appl., 119 (2009),
  pp.~317--333.

\bibitem{Kristaly2010}
{\sc A.~Krist\'{a}ly}, {\em {Location of Nash equilibria: a Riemannian
  geometrical approach}}, Proc. Amer. Math. Soc., 138 (2010), pp.~1803--1810.

\bibitem{Kristaly2014}
\leavevmode\vrule height 2pt depth -1.6pt width 23pt, {\em {Nash-type
  equilibria on Riemannian manifolds: A variational approach}}, J. Math. Pures
  Appl., 101 (2014), pp.~660--688.

\bibitem{KLLN2015}
{\sc A.~Krist\'{a}ly, C.~Li, G.~Lopez, and A.~Nicolae}, {\em {What do
  ``convexities" imply on Hadamard manifolds?}}, J. Optim. Theory Appl., 170
  (2016), pp.~1068--1079.

\bibitem{Li2009}
{\sc C.~Li, G.~L\'{o}pez, and V.~Mart\'{\i}n-M\'{a}rquez}, {\em {Monotone
  vector fields and the proximal point algorithm on Hadamard manifolds}}, J.
  Lond. Math. Soc., 79 (2009), pp.~663--683.

\bibitem{LiMWY2011}
{\sc C.~Li, B.~S. Mordukhovich, J.~Wang, and J.~C. Yao}, {\em {Weak sharp
  minima on Riemannian manifolds}}, SIAM J. Optim., 21 (2011), pp.~1523--1560.

\bibitem{LiY2012}
{\sc C.~Li and J.~C. Yao}, {\em {Variational inequalities for set-valued vector
  fields on Riemannian manifolds: convexity of the solution set and the
  proximal point algorithm}}, SIAM J. Control Optim., 50 (2012),
  pp.~2486--2514.

\bibitem{LiLi2009}
{\sc S. L. Li, C. Li, Y. C. Liou, and J. C. Yao}, {\em {Existence of solutions for
  variational inequalities on Riemannian manifolds}}, Nonlinear Anal., 71
  (2009), pp.~5695--5706.

\bibitem{Mahony1996}
{\sc R.~E. Mahony}, {\em {The constrained Newton method on Lie group and the
  symmetric eigenvalue problem}}, Linear Algebra Appl., 248 (1996), pp.~67--89.

\bibitem{Miller2005}
{\sc S.~A. Miller and J.~Malick}, {\em {Newton methods for nonsmooth convex
  minimization: Connections among U-Lagrangian, Riemannian Newton and SQP
  methods}}, Math. Program., 104 (2005), pp.~609--633.

\bibitem{Morgan2007}
{\sc J.~Morgan and V.~Scalzo}, {\em {Pseudocontinuous functions and existence
  of Nash equilibria}}, J. Math. Econom., 43 (2007), pp.~174--183.

\bibitem{Nash1950}
{\sc J.~Nash}, {\em {Equilibrium points in $n$-person games}}, Pro. Nat. Acad.
  Sci. U.S.A., 36 (1950), pp.~48--49.

\bibitem{Nash1951}
\leavevmode\vrule height 2pt depth -1.6pt width 23pt, {\em {Non-cooperative
  games}}, Ann. of Math., 54 (1951), pp.~286--295.

\bibitem{Nessah2008}
{\sc R.~Nessah and K.~Kerstens}, {\em {Characterizations of the existence of
  Nash equilibria with non-convex strategy sets}}, Working Papers 2008-ECO-13,
  IESEG School of Management,  (2008).

\bibitem{Oettli1997}
{\sc W.~Oettli}, {\em {A remark on vector-valued equilibria and generalized
  monotonicity}}, Acta Math. Vietnam., 21 (1997), pp.~215--221.

\bibitem{Rockafellar1970}
{\sc R.~T. Rockafellar}, {\em {Convex Analysis}}, Princeton University Press,
  Princeton, 1970.

\bibitem{Rosen1965}
{\sc J.~B. Rosen}, {\em {Existence and uniqueness of equilibrium points for
  concave Nperson games}}, Econometrica, 33 (1965), pp.~520--534.

\bibitem{Sakai1996}
{\sc T.~Sakai}, {\em {Riemannian Geometry}}, Transl. Math. Monogr., AMS, Providence, 1996.

\bibitem{Smith1994}
{\sc S.~T. Smith}, {\em Geometric Optimization Methods for Adaptive Filtering},
  PhD Thesis, Harvard University Cambridge Massachusetts, 1994.

\bibitem{Tala1996}
{\sc J. E. Tala and E.~Marchi}, {\em {Games with non-convex strategy sets}},
  Optimization, 37 (1996), pp.~177--181.

\bibitem{Udriste1994}
{\sc C.~Udriste}, {\em {Convex Functions and Optimization Methods on Riemannian
  Manifolds}}, Math. Appl., Springer, New York, 1994.

\bibitem{Wang2010}
{\sc J.~H. Wang, G.~L\'{o}pez, V.~Mart\'{\i}n-M\'{a}rquez, and C.~Li}, {\em
  {Monotone and accretive vector fields on Riemannian manifolds}}, J. Optim.
  Theory Appl., 146 (2010), pp.~691--708.

\bibitem{WangYangHe2001}
{\sc S. L. Wang, H.~Yang, and B.S. He}, {\em {Inexact implicit method with
  variable parameter for mixed monotone variational inequalities}}, J. Optim.
  Theory Appl., 111 (2001), pp.~431--443.

\bibitem{Yang2007}
{\sc Y.~Yang}, {\em {Globally convergent optimization algorithms on Riemannian
  manifolds: uniform framework for unconstrained and constrained
  optimization}}, J. Optim. Theory Appl., 132 (2007), pp.~245--265.

\bibitem{Bacak2013}
{\sc M.~Bacak}, {\em {The proximal point algorithm in metric spaces}}, Israel Journal of Mathematics, 194 (2013), pp.~689--701.

\bibitem{Batista2015}
{\sc E. E. A.~Batista, G. C.~Bento, and O. P.~Ferreira}, {\em {An existence result for the generalized vector equilibrium problem on Hadamard manifolds}}, J. Optim. Theory Appl., 167 (2015), pp.~550--557.

\bibitem{Ferreira2002}
{\sc P. R.~Oliveira, and O. P.~Ferreira}, {\em {Proximal Point Algorithm on Riemannian Manifolds}}, Optimization, 51 (2002), pp.~257--270.

\bibitem{Nemeth2003}
{\sc S. Z.~Nemeth}, {\em {Variational inequalities on Hadamard manifolds}}, Nonlinear Anal., 52(5) (2003), pp.~1491--1498.

\bibitem{Afsari2010}
{\sc B.~Afsari}, {\em {Riemannian $L^p$ center of mass: existence, uniqueness, and
  convexity}}, Proc. Amer. Math. Soc., 139(2) (2010), pp.~655--673.

\end{thebibliography}

\end{document}